\documentclass[journal]{IEEEtran}
\usepackage{cite}
\usepackage{amsmath,amssymb,amsfonts,amsthm}
\usepackage{graphicx}
\usepackage{booktabs}
\usepackage{longtable}
\usepackage{rotating}
\usepackage{adjustbox}
\usepackage{multirow}
\usepackage{array}
\usepackage{url}
\usepackage{xcolor}
\usepackage{placeins}
\usepackage{ifthen}
\usepackage[pass]{geometry}
\usepackage{pdflscape}
\usepackage[hidelinks]{hyperref}
\usepackage{pgfplots}\pgfplotsset{compat=1.18}
\graphicspath{{figs/}}

\newtheorem{remark}{Remark}
\newtheorem{lemma}{Lemma}
\newtheorem{corollary}{Corollary}
\newcommand{\ipoptopt}[1]{\begingroup\urlstyle{tt}\nolinkurl{#1}\endgroup}

\newif
\ifextended 
\extendedtrue 

\newcommand{\extendedcite}{~\cite{taheri2026warmstart_extended}}

\ifextended 
    \newcommand{\appProofs}{Appendix~A} 
    \newcommand{\appSuppSummaries}{Appendix~B} 
    \newcommand{\appBlockDiagnostics}{Appendix~C} 
    \newcommand{\appDualDCScaling}{Appendix~D} 
    \newcommand{\appRankings}{Appendix~E} 
\else 
    \newcommand{\appProofs}{Appendix~A of the extended version\extendedcite} \newcommand{\appSuppSummaries}{Appendix~B of the extended version\extendedcite} \newcommand{\appBlockDiagnostics}{Appendix~C of the extended version\extendedcite} \newcommand{\appDualDCScaling}{Appendix~D of the extended version\extendedcite} \newcommand{\appRankings}{Appendix~E of the extended version\extendedcite} 
\fi

\title{Not All Warm Starts Help: Benchmarking Primal-Dual Initializations for ACOPF Algorithms}
\author{Babak Taheri and Daniel K. Molzahn%
\thanks{Babak Taheri is with Hitachi Energy Research, Raleigh, NC 27606 USA. Daniel K. Molzahn is with the School of Electrical and Computer Engineering, Georgia Institute of Technology, Atlanta, GA 30332 USA (emails: babak.taheri@hitachienergy.com, molzahn@gatech.edu).}}

\begin{document}
\maketitle

\begin{abstract}
Warm starts are widely used to accelerate AC optimal power flow (ACOPF) solves, but the impact of different initialization strategies has received limited systematic study, particularly for the primal-dual interior-point methods that dominate large-scale ACOPF algorithms. This paper benchmarks initialization strategies for ACOPF solved with the interior-point solver IPOPT on 19 PGLib-OPF instances (5 to 30{,}000 buses), testing all 15 non-empty subsets of the primal blocks $\{P_g, Q_g, V_m, V_a\}$ under oracle conditions and three DC-seeded combinations in a practical setting. The experiments show that most partial primal-plus-dual restarts increase solve time or reduce convergence reliability. Among the oracle primal-plus-dual (O-PD) configurations, only the complete restart reliably converges on every baseline-convergent case, reaching a $47.6\%$ median solve-time speedup. Twelve of the 14 partial O-PD combinations have negative median speedups, and several fail repeatedly on larger networks. Decomposing the dual into constraint and bound multipliers shows that \emph{coverage}, not the presence of duals per se, governs robustness: the full bound-multiplier vector reaches 90.7\% convergence and a $+26.8$\% median speedup, whereas block-matched coverage (oracle multipliers on some bounds, defaults on the rest) drops to 70.4\% and $-31.1$\%. Practical DC seeding sometimes helps the AC solve, but the benefit is no longer statistically significant once the DCOPF presolve cost is included in the end-to-end comparison ($p = 0.4171$). For learned warm-start methods, the results support the following target ordering: predict the full primal vector first; if only partial coverage is possible, prioritize voltage variables; and avoid partial or inconsistent dual predictions unless the primal estimate is nearly complete.
\end{abstract}


\section{Introduction}
\label{sec:intro}

Initialization matters in AC optimal power flow (ACOPF) solved with nonlinear interior-point methods. On large problems, the starting point influences the first barrier iterates, line-search behavior, restoration steps, and wall-clock time~\cite{nocedal2006numerical,wachter2006ipopt,wu2001warmstart}. In practice, those initialization choices are often discussed less explicitly than the modeling stack or solver settings, which makes it harder to tell whether a runtime improvement comes from the algorithm itself or from the way it is started. This gap matters most for ML-based warm-start methods, which must choose their prediction targets at training time without knowing which variables are safe or beneficial to predict.

This paper benchmarks ACOPF initialization strategies across 19 PGLib-OPF instances spanning 5 to 30{,}000 buses. The benchmarking process uses one AC modeling stack and one baseline case-data start, evaluates all $2^4 - 1 = 15$ non-empty subsets of the primal block set 
$\{P_g,Q_g,V_m,V_a\}$ (generator active power $P_g$, generator reactive power $Q_g$, bus voltage magnitude $V_m$, and bus voltage angle $V_a$) for the oracle AC families, and a practical DC-seeded family restricted to the three native DC combinations $\{P_g\}$, $\{V_a\}$, and $\{P_g,V_a\}$.
The benchmark also separates constraint multipliers from bound multipliers in the dual initialization and keeps AC solve time distinct from end-to-end (E2E) workflow time when DC presolve is involved. The main goal is to identify which primal and dual variables are worth initializing, and which ones tend to make the solve slower or less reliable.

A key motivation for this benchmark is to inform the design of machine-learning (ML) models that predict warm-start points for ACOPF solvers. 
Recent work has trained neural networks to predict primal or primal-dual ACOPF solutions~\cite{pan2023deepopf,zamzam2020learning}, used neural-network-predicted warm-start points to seed conventional solvers~\cite{baker2019learning}, and applied decision-tree methods for the same purpose~\cite{cao2023dtwarm}, yet the sensitivity of solver behavior to the choice of prediction targets has not been systematically quantified. The oracle experiments quantify this sensitivity, while the DC-seeded runs show what a physics-based predictor can offer without learning. The four initialization classes studied (primal-only, primal-plus-dual, dual-only, and practical DC-seeded) are defined in Section~\ref{sec:problem}. Throughout, ``oracle'' denotes a same-instance protocol in which the converged AC solution is available by construction.

\subsection*{Contributions}
\begin{enumerate}
    \item Isolate the effect of dual-initialization \emph{coverage}, showing that supplying the full bound-multiplier vector is substantially more robust than block-matched coverage (90.7\% vs.\ 70.4\% convergence; $+26.8$\% vs.\ $-31.1$\% median speedup).
    \item Provide the first subset-level benchmark over all 15 non-empty primal-block subsets and the three native DC-seeded combinations on 19 PGLib-OPF cases (5--30,000 buses), finding that 12 of 14 partial primal-plus-dual restarts are harmful in median terms.
    \item Separate AC solve time from end-to-end time and show that the one-shot DC-seeding benefit, present in AC solve time, is not significant once presolve cost is charged ($p=0.4171$).
    \item Translate the oracle results into ceiling-level target priorities for learned warm starts, framed as upper-bound hypotheses rather than out-of-sample claims.
\end{enumerate}
\section{Related Work}
\label{sec:related}

ACOPF pairs an economic dispatch objective with nonlinear AC network physics and operating constraints, and has long served as a standard benchmark in power systems~\cite{zimmerman2011matpower,coffrin2019pglib}. The problem is nonconvex and large scale, and its numerical sensitivity makes it a natural testbed for solver engineering and performance evaluation.
 
Primal-dual interior-point methods remain the standard for large-scale ACOPF, with IPOPT as a widely used reference implementation~\cite{wachter2006ipopt}. The general theory of interior-point methods and their sensitivity to initialization is well established~\cite{nocedal2006numerical}. Wu and Debs~\cite{wu2001warmstart} provided early empirical evidence that starting-point quality strongly influences convergence in nonlinear interior-point OPF solvers. Gondzio and Grothey~\cite{gondzio2003warmstart} analyzed LP reoptimization strategies and showed that multiplier information is beneficial when the primal state is close to the new optimum and potentially harmful otherwise. Their work, though restricted to LP, provides analogous theoretical intuition for the mechanism studied empirically here in ACOPF.
For nonlinear programs, Forsgren~\cite{forsgren2005} observes that interior-point iteration counts are relatively insensitive to the starting point, the very property that makes these methods difficult to warm-start, and shows that warm-starting from a near-optimal solution to a nearby problem is closely related to an SQP step on an equality-constrained problem.
 
DCOPF-based seeding is computationally inexpensive and supplies physically meaningful active-power and angle information~\cite{coffrin2014linear,stott2009dc}. DCOPF omits reactive-power balance and assumes per-unit voltages at unity, so $V_m$ and $Q_g$ are structurally unavailable from a standard DCOPF solve.
 
On the machine-learning side, two directions have emerged: predicting near-feasible or near-optimal ACOPF solutions directly, bypassing the conventional solver~\cite{pan2023deepopf,zamzam2020learning}, and predicting warm-start points to seed a conventional solver for faster convergence~\cite{baker2019learning}. Decision-tree-based warm-start methods have also been reported for ACOPF~\cite{cao2023dtwarm}. A
related concurrent work, the WARP benchmark~\cite{suri2026warp}, establishes that primal-only gains against a flat start largely vanish against IPOPT's $(l+u)/2$ default and that the complete primal-dual-barrier state is required for large oracle speedups. Our study addresses a question it leaves open: which subsets of the primal and dual blocks drive that behavior. We enumerate all 15 non-empty primal-block subsets and decompose the dual initialization into constraint multipliers, block-matched bound multipliers, and full bound-multiplier coverage, revealing that full bound-multiplier coverage is markedly more robust than block-matched coverage, a distinction WARP does not isolate. We report wall-clock behavior up to 30,000 buses against the case-data start practitioners typically deploy.
A complementary line of work learns dual information directly: dual conic proxies predict feasible dual solutions to convex relaxations of ACOPF, for both second-order-cone~\cite{qiu2024dualconic} and semidefinite~\cite{qiu2025dualsdp} relaxations, to produce valid optimality certificates. Those methods predict duals of a \emph{relaxed} problem to bound the optimum rather than to seed the AC solve, but they are relevant to the present question of which dual variables, if any, should be predicted for warm starts.
The oracle results quantify the same-instance restart behavior that warm-start predictors could approach under exact prediction, while also identifying variable subsets associated with lower failure rates and shorter solve times.

\section{Problem Setting and Initialization Families}
\label{sec:problem}

\subsection{ACOPF Formulation}

Let the ACOPF decision vector be
\begin{equation}
  x := (V_m,\; V_a,\; P_g,\; Q_g),
\end{equation}
where $V_m \in \mathbb{R}^{n_b}$ and $V_a \in \mathbb{R}^{n_b}$ are bus voltage magnitudes and angles, and $P_g \in \mathbb{R}^{n_g}$ and $Q_g \in \mathbb{R}^{n_g}$ are generator active and reactive outputs. The benchmarked ACOPF is
\begin{align}
  \min_x \quad & f(P_g) \label{eq:obj}\\
  \text{s.t.} \quad & g(x) = 0, \label{eq:eq}\\
                    & c(x) \le 0, \label{eq:ineq}\\
                    & l \le x \le u, \label{eq:bounds}
\end{align}
where $f(\cdot)$ is the generator cost function, $g(\cdot)$ encodes the nonlinear AC power-balance equations together with the reference-angle condition $V_a[\text{ref}] = 0$, $c(\cdot)$ collects general inequality constraints such as branch thermal limits, and $l \le x \le u$ collects the simple variable bounds present in the implemented ACOPF model. The separation of general inequality constraints~\eqref{eq:ineq} from simple bounds~\eqref{eq:bounds} reflects IPOPT's internal treatment: constraint multipliers $\lambda$ and $\mu_c$ are associated with $g$ and $c$, while distinct bound multipliers $z_L$ and $z_U$ are associated only with those components of $x$ that carry explicit finite lower and/or upper bounds in the implementation. This matters for the dual decomposition in Section~\ref{sec:dual_decomp}.

The associated Lagrangian is
\begin{align}
\label{eq:lagrangian}
  \mathcal{L}(x, \lambda, \mu_c, z_L, z_U) :=
    f(P_g) + \lambda^\top g(x) + \mu_c^\top c(x) \nonumber \\
    - z_L^\top (x - l) - z_U^\top (u - x),
\end{align}
where $\mu_c, z_L, z_U \ge 0$ at a KKT point. IPOPT solves a sequence of barrier subproblems parameterized by decreasing $\mu > 0$ via a primal-dual Newton method, making the algorithm sensitive to both the primal starting point $(V_m^0, V_a^0, P_g^0, Q_g^0)$ and the dual starting point $(\lambda^0, \mu_c^0, z_L^0, z_U^0)$~\cite{wachter2006ipopt,nocedal2006numerical}. For notational compactness in subsequent analysis, we write $\nu := (\lambda, \mu_c, z_L, z_U)$ for the full dual vector.

\subsection{Interior-Point Warm-Start Consistency}
\label{sec:ipm_init}

Interior-point warm starts are delicate because the supplied primal and dual values must be compatible with the barrier subproblem, not only with the original ACOPF KKT system. IPOPT pushes primal variables and bound multipliers into the interior using its warm-start bound-push parameters, which are set to $10^{-6}$ in this study (Section~\ref{sec:compute}). Therefore, a partial restart can be unstable when oracle values are supplied on some blocks while omitted primal blocks remain at baseline values and omitted bound multipliers retain IPOPT-default initialization. The resulting stationarity, complementarity, and barrier-centering mismatches are formalized in Section~\ref{sec:oracle_pd} (Lemma~\ref{lem:residual} and Corollaries~\ref{cor:compl}--\ref{cor:barrier}). A full barrier interpretation is provided in \appProofs{}.

\subsection{Baseline Case-Data Start}
\label{sec:baseline}

The baseline uses native MATPOWER or PGLib-OPF case data exactly as provided, with the four primal blocks read directly from each case file without modification. This case-data initialization serves as the reference for all speedup computations in this study; it is the start practitioners typically deploy, whereas IPOPT's $(l+u)/2$ midpoint is an alternative reference baseline. One case in the suite (8,387 buses) fails to converge under the baseline solver configuration and is therefore termed \emph{baseline-nonconvergent}. Because the downstream warm-start experiments in this protocol are constructed from the baseline solve state, that case is excluded from downstream warm-start comparisons and from all medians, speedup computations, and statistical tests, while still being retained in the overall case count.
\subsection{Oracle AC Primal-Only Restart}
\label{sec:oracle_po}

In the oracle primal-only (O-PO) family, the target ACOPF is solved once from the baseline start, and a specified subset of $\{P_g, Q_g, V_m, V_a\}$ from the converged solution is reused as the initial primal point for a second solve of the same problem. All 15 non-empty subsets are tested. Because no dual variables are supplied in this family, the primal-dual inconsistency mechanism discussed in Section~\ref{sec:oracle_pd} does not arise, although partial primal starts can still be infeasible or poorly centered. This leads to the more benign pattern discussed in Section~\ref{sec:blocks}.

\subsection{Oracle AC Primal-Plus-Dual Restart Families}
\label{sec:oracle_pd}

The oracle primal-plus-dual families augment O-PO by also supplying IPOPT's dual variables from the first solve. Because the dual vector comprises two structurally distinct components, constraint multipliers $(\lambda^\star, \mu_c^\star)$ for the equality and inequality constraints and bound multipliers $(z_L^\star, z_U^\star)$ for the variable bounds, the dual initialization is decomposed into four modes:

\begin{enumerate}
  \item \textbf{Block-matched primal-plus-dual (O-PD):} Supplies constraint multipliers $(\lambda^\star, \mu_c^\star)$ and the bound multipliers associated with the bound-constrained variables in the initialized primal blocks.
  \item \textbf{Constraint-dual-only (O-CD):} Supplies constraint multipliers $(\lambda^\star, \mu_c^\star)$ without any bound multipliers.
  \item \textbf{Bounds-dual-only (O-BD):} Supplies the corresponding block-matched bound multipliers without constraint multipliers.
  \item \textbf{All-bounds primal-plus-dual (O-PD-AB):} Supplies constraint multipliers $(\lambda^\star, \mu_c^\star)$ together with the full bound-multiplier vector for all bound-constrained variables, not just those associated with the initialized blocks.
\end{enumerate}

All four modes are tested across all 15 non-empty primal-block subsets, yielding $15 \times 4 = 60$ dual-augmented combinations (plus the 15 primal-only combinations). Supplying dual variables alongside an \emph{incomplete} primal set creates a primal-dual inconsistency whose severity depends on which dual components are present.
Throughout, $\|\cdot\|$ denotes the Euclidean ($\ell_2$) norm for vectors and the induced spectral norm for matrices.

\begin{lemma}[Stationarity residual under partial primal replacement]
\label{lem:residual}
Let $(x^\star, \nu^\star)$ satisfy the first-order KKT conditions of the ACOPF with Lagrangian~\eqref{eq:lagrangian}, so that the stationarity residual
\begin{equation}
  r(x) := \nabla_x \mathcal{L}(x, \nu^\star)
\end{equation}
satisfies $r(x^\star) \approx 0$ to solver tolerance. A partial restart supplies $\tilde{x}$ with $\tilde{x}_S = x^\star_S$ for the initialized block set $S$ and $\tilde{x}_{\bar{S}} = x^0_{\bar{S}}$ for the remaining blocks $\bar{S} \neq \emptyset$. Then
\begin{equation}
\label{eq:sensitivitybound}
  \|r(\tilde{x})\| \le \|r(x^\star)\|
  + M \cdot \|\tilde{x} - x^\star\|,
\end{equation}
where $M := \sup_{\xi \in [x^\star,\, \tilde{x}]} \bigl\|\nabla^2_{xx} \mathcal{L}(\xi, \nu^\star)\bigr\|$ is finite because the segment $[x^\star,\, \tilde{x}]$ is compact and $\nabla^2_{xx}\mathcal{L}(\cdot, \nu^\star)$ is continuous by $C^2$ smoothness of $f$, $g$, $c$, and the affine bound terms. In ACOPF, the affine bound terms contribute nothing to this Hessian, which is determined by $\nabla^2 f$ together with the nonlinear AC network equations in $g(x)$ and $c(x)$. In particular, the full restart ($\bar{S} = \emptyset$) trivially yields $\|\tilde{x} - x^\star\| = 0$ and recovers $\|r(\tilde{x})\| \approx 0$.
\end{lemma}

\begin{proof}
By the fundamental theorem of calculus along the segment from $x^\star$ to $\tilde{x}$,
\begin{equation}
  r(\tilde{x}) - r(x^\star)
  = \int_0^1 \nabla_x r\!\bigl(x^\star + t(\tilde{x} - x^\star)\bigr)\,(\tilde{x} - x^\star)\,dt.
\end{equation}
Since $\nabla_x r = \nabla^2_{xx}\mathcal{L}$, taking norms and applying the triangle inequality gives \eqref{eq:sensitivitybound}. Because $\tilde{x}_{\bar{S}} \neq x^\star_{\bar{S}}$ in general and $r(x^\star) \approx 0$ to solver tolerance, $\|r(\tilde{x})\|$ is controlled by the residual at $x^\star$ together with an $M$-scaled displacement term for the omitted blocks.
\end{proof}

\begin{corollary}[Complementarity gap under partial primal replacement]
\label{cor:compl}
Under the setup of Lemma~\ref{lem:residual}, the converged solution satisfies the complementarity conditions $(x^\star_i - l_i)\,z_{L,i}^\star \approx 0$ and $(u_i - x^\star_i)\,z_{U,i}^\star \approx 0$ for all components $i$. For any component $i \in \bar{S}$ with $\tilde{x}_i \neq x^\star_i$:
\begin{equation}
  (\tilde{x}_i - l_i)\,z_{L,i}^\star
  = \underbrace{(x^\star_i - l_i)\,z_{L,i}^\star}_{\approx\, 0}
  + (\tilde{x}_i - x^\star_i)\,z_{L,i}^\star.
\end{equation}
If $z_{L,i}^\star > 0$, then the lower bound $x_i \ge l_i$ is active at optimality, and the second term is generically nonzero, producing a complementarity residual of order $|\tilde{x}_i - x_i^\star| \cdot z_{L,i}^\star$. An analogous statement holds for upper-bound multipliers $z_U^\star$.
\end{corollary}

\begin{corollary}[Barrier-centering gap under reused bound multipliers]
\label{cor:barrier}
Under the setup of Corollary~\ref{cor:compl}, let IPOPT restart with barrier parameter $\mu_{\text{init}} > 0$. For any lower-bound multiplier component $z_{L,i}^\star$ reused together with a primal component whose initial value differs from $x_i^\star$,
\begin{equation}
  (\tilde{x}_i - l_i) z_{L,i}^\star - \mu_{\text{init}}
  = \bigl((x_i^\star - l_i) z_{L,i}^\star - \mu_{\text{init}}\bigr)
  + (\tilde{x}_i - x_i^\star) z_{L,i}^\star.
\end{equation}
An analogous identity holds for the upper-bound term $(u_i-\tilde{x}_i) z_{U,i}^\star - \mu_{\text{init}}$. When $|\tilde{x}_i - x_i^\star| \cdot z_{L,i}^\star \gg |(x_i^\star-l_i)z_{L,i}^\star-\mu_{\text{init}}|$, the displacement term dominates the barrier-centering residual. Thus, even when the saved multipliers are locally consistent with the near-final point of the first solve, partial primal replacement generically perturbs IPOPT's barrier-centering equations by an amount proportional to the omitted-block displacement times the reused bound multiplier.
\end{corollary}

Partial primal-plus-dual starts incur a stationarity residual bounded by the omitted-block displacement (Lemma~\ref{lem:residual}), together with complementarity and barrier-centering inconsistencies that are generically nonzero (Corollaries~\ref{cor:compl}--\ref{cor:barrier}), especially when $\mu_{\text{init}} = 10^{-6}$ (Section~\ref{sec:compute}). Corollaries~\ref{cor:compl} and~\ref{cor:barrier} are stated for components whose bound multiplier is reused at its optimal value $z^\star$; this corresponds to the all-bounds (O-PD-AB) and dual-only modes. In the block-matched O-PD mode the omitted-block bound multipliers instead retain IPOPT defaults, so block-matched restarts carry the additional inconsistency that the reused optimal constraint multipliers $(\lambda^\star,\mu_c^\star)$ no longer match those default bound multipliers; the empirical ordering O-PD-AB $>$ O-PD in Section~\ref{sec:dual_decomp} is consistent with this additional term being harmful. This additional stationarity perturbation is made precise in \appProofs{}. The issue is therefore not that bound multipliers are intrinsically harmful. Rather, a partial restart mixes oracle primal and dual information on some blocks with baseline primal values and default dual initialization on others, producing an initial state that is not mutually consistent with the KKT and barrier-centering equations. These results are a local inconsistency rationale rather than a proof of a specific IPOPT failure mode.

\subsection{Dual-Only Control Experiments}
\label{sec:dual_only}

To probe the effect of dual information in the absence of any primal improvement, three dual-only experiments supply dual variables from the converged solution while leaving all four primal blocks at baseline values: constraint multipliers only (DO-C), all bound multipliers only (DO-B), and the full dual vector (DO-F). These are ablation tests that help characterize the role of each dual component; they are not intended as practical initialization strategies.

\subsection{Practical DC-Seeded Starts}
\label{sec:dc}

The practical family begins with a DCOPF solve. Only $P_g$ and $V_a$ are available from DCOPF; $V_m$ and $Q_g$ remain at baseline case-data values throughout. A single DCOPF computation produces both quantities; the three tested combinations are:
\begin{equation}
  \{P_g\}, \quad \{V_a\}, \quad \{P_g, V_a\}.
\end{equation}
The DCOPF presolve cost is common to all three, so differences in E2E times across combinations reflect only differences in AC solve time and case convergence.

\section{Experimental Protocol}
\label{sec:protocol}

\subsection{Benchmark Suite}

The study comprises 19 benchmark instances from PGLib-OPF~\cite{coffrin2019pglib} spanning 5 to 30,000 buses and 6 to 35,393 branches, covering small instructional networks (PJM, IEEE, EPRI), medium IEEE and ACTIV systems, and large PEGASE, GOC, RTE, and EpiGrids instances. As noted in Section~\ref{sec:baseline}, one case (8,387 buses) is baseline-nonconvergent and is retained in the overall case count ($n=19$) but excluded from all downstream comparisons, giving an effective sample of $n=18$.
\subsection{Recorded Metrics and Timing Conventions}

Four quantities are recorded per run:
\begin{enumerate}
  \item \textbf{AC solve time}: wall-clock IPOPT solve time, excluding model construction.
  \item \textbf{Total time}: model build plus ACOPF solve time.
  \item \textbf{E2E time}: complete one-shot workflow time, including DCOPF presolve for DC-seeded methods; equal to total time for oracle and baseline methods.
  \item \textbf{Solve-time speedup}: $(t_{\text{base}} - t_{\text{method}}) / t_{\text{base}} \times 100\%$.
\end{enumerate}

Two conventions require emphasis. First, median speedup values are medians of per-case percentage improvements, not the percentage derived from the ratio of suite-level median times; these are different statistics and need not agree in sign. Second, all timing measurements are single-run wall-clock values. For the smallest cases (baseline solve times of 20--40~ms), OS scheduling jitter can dominate, so speedup figures for cases with baseline solve times below approximately 100~ms should be interpreted as indicative rather than precise. Accordingly, all conclusions are framed in terms of the sign and ranking of effects on medium- and large-scale cases, where relative jitter is small, rather than the precise magnitude of any individual speedup; per-case percentages for sub-100\,ms cases are reported for completeness but support no claim.

A run is classified as failed~(F) if the solver termination condition is not reported as \texttt{optimal}, \texttt{locallyOptimal}, or, when available, \texttt{globallyOptimal}. Failed runs are excluded from all medians, means, and speedup computations. In this benchmark, all observed failures terminated at the iteration limit (300 iterations) rather than with an infeasibility declaration, which is more consistent with slow or stalled convergence than with an explicit infeasibility diagnosis.

\subsection{Static vs.\ Case-Wise Best Views}

The static ranking applies one fixed policy uniformly to all cases: ``What single strategy should a practitioner deploy without case-specific tuning?'' When selecting the best static policy within a family, the criterion is minimum median AC solve time over converged cases. The case-wise best selects the best-performing method within a family for each case \emph{ex post}: ``What is the oracle upper bound of this initialization family?'' Case-wise best rows are upper-bound diagnostics, not prescriptions for deployment. Accordingly, significance markers for case-wise best rows should be interpreted only as descriptive summaries of the selected envelopes, not as confirmatory tests of a pre-specified method.

A subtlety exposed by this benchmark is that the minimum-median-solve-time criterion can select policies with low convergence rates: if a policy converges only on smaller, faster cases, its median over those cases may be low even though it fails on the harder instances that dominate practical workloads. This survivor-bias effect is visible in Table~\ref{tab:global_summary} for the best static O-PD policy and, more mildly, for the best static O-PO policy as well.

\subsection{Statistical Testing}
\label{sec:stats}

Matched-case pairwise comparisons use the two-sided Wilcoxon signed-rank test~\cite{wilcoxon1945} on pairs of cases where both methods converge. Because the 8,387-bus case is \emph{baseline-nonconvergent} under the study protocol and is excluded from downstream warm-start comparisons, the effective sample is $n = 18$ for most comparisons; comparisons involving the bounds-dual-only family (which fails on additional cases) have smaller effective samples as noted in Table~\ref{tab:pairwise_tests}.

Unless a smaller matched set is reported, these tests therefore include the smallest baseline-convergent cases as well as the medium and large instances. Because all timings are single-run measurements, the tests assess systematic differences \emph{across the case population}, not timing variance under repeated execution of any individual case. The qualitative interpretation is anchored by the medium- and large-case summaries as well as by the full-sample $p$-values. Twelve family-level contrasts are tested simultaneously; both Holm sequential correction~\cite{holm1979} and Bonferroni correction are reported at $\alpha = 0.05$. The Hodges--Lehmann estimator~\cite{hodges1963estimates} of the median paired difference is reported as an effect-size complement to each $p$-value. All statistical computations used SciPy~1.11.2~\cite{scipy2020}.

\subsection{Computational Environment}
\label{sec:compute}

All experiments were conducted on a MacBook Pro with Apple M4 chip and 16\,GB RAM, running macOS. ACOPF runs used IPOPT single-threaded with the MUMPS~5.5.1~\cite{amestoy2001mumps} linear solver, also single-threaded, to ensure wall-clock times reflect sequential solver behavior. The software stack comprised Python~3.11.4, Pyomo~6.6.2~\cite{hart2017pyomo}, SciPy~1.11.2~\cite{scipy2020}, and IPOPT~3.14.13 built from source (MUMPS only). All ACOPF problems were solved with IPOPT; the DCOPF presolve in the DC-seeded family was solved with HiGHS~\cite{huangfu2018highs}. The released code is solver-agnostic: it selects the DCOPF solver automatically from those installed, with IPOPT as a fallback, and the IPOPT linear solver is whichever the IPOPT build provides. Thus, either step can target a different installed solver without code changes. We used HiGHS and MUMPS here; Gurobi for DCOPF and HSL linear solvers such as MA27, MA57, or MA97 for IPOPT may be used when available. ACOPF runs used convergence tolerance $10^{-6}$ and maximum iteration limit~300. All families that supply dual information set
\ipoptopt{warm_start_init_point=yes},
\ipoptopt{warm_start_bound_push=1e-6},
\ipoptopt{warm_start_bound_frac=1e-6}, and
\ipoptopt{warm_start_mult_bound_push=1e-6}, with
$\mu_{\text{init}} = 10^{-6}$~\cite{wachter2006ipopt,ipopt_options}.\footnote{Code repository: \url{https://github.com/BabakTaheri1/acopf_warmstart_benchmark}}

For oracle experiments, the Pyomo model is rebuilt from scratch for each run from a shared case-data context (precomputed network admittance matrices and cost structures), providing partial isolation from in-process caching effects, but both solves (the reference solve and the warm-started solve) execute within the same OS process. Because any residual second-solve advantage (e.g., from OS-level caching or memory layout) would favor the warm-started run, the reported oracle speedups should be read as upper bounds on the pure initialization effect. For all converged warm-start runs, the returned objective value was verified to match the baseline objective value to within the study tolerance, which is consistent with convergence to the same local optimum as the baseline but does not by itself prove identity of the local solution or KKT point.
A fully controlled baseline-vs-baseline re-solve experiment to quantify this confound was not conducted; see Section~\ref{sec:discussion} for further discussion.


\section{Overall Quantitative Results}
\label{sec:global}

Table~\ref{tab:global_summary} summarizes the benchmark at the family level. The static rows show what happens when one fixed policy is used across all cases. The case-wise best rows show the best outcome each family could achieve if the winning combination were chosen separately for each case after the fact. Within the block-matched O-PD family, the robust fixed policy is the full-vector $P_g+Q_g+V_m+V_a$ restart, which converges on all 18 baseline-convergent cases. In the case-wise best view, the best O-PD envelope reaches a $47.6\%$ median AC solve-time speedup, while the best O-PD-AB envelope reaches $50.3\%$. That $50.3\%$ figure is an ex-post upper bound; the best fixed O-PD-AB policy in the full results matrix is the full-vector $P_g{+}Q_g{+}V_m{+}V_a$ restart at $50.1\%$.

The complete case-level warm-start matrix is provided in Table~\ref{tab:all_warmstarts_matrix}.

The O-PD results show a clear separation between the reliable full restart and the degraded performance of most partial restarts. In this family, 12 of the 14 partial combinations have negative median speedups. The two combinations that remain positive in median terms, $P_g{+}V_m{+}V_a$ ($+6.1\%$) and $P_g{+}Q_g{+}V_m$ ($+3.2\%$), converge on only 15/18 and 14/18 baseline-convergent cases, respectively. Several one- and two-block combinations have median slowdowns below $-60\%$ and repeated failures on larger instances.

The DC-seeded family shows a weaker and less consistent effect. In the case-wise best view, the best DC seed achieves a $12.5\%$ AC solve-time speedup. That effect appears in the matched-case comparison against baseline solve time (raw $p = 0.0237$), but it does not survive Holm or Bonferroni correction. Once the DC presolve cost is included, the comparison against baseline total time is not statistically significant (raw $p = 0.4171$). Thus, DC information can reduce AC solve time, but in this one-shot benchmark the reduction does not generally offset the DCOPF presolve cost.

The static ranking also illustrates why convergence must be read alongside medians.
The minimum-median-solve-time rule selects $V_m$ as the best static O-PD policy (median solve 0.270~s over its 10 converged cases), yet that policy converges on only 10 of 18 baseline-convergent cases; its median speedup over those same 10 cases is $-93.9\%$, reflecting that even the cases it solves are typically slower than the baseline. In this family, the robust fixed policy is the full-vector restart, not the minimum-median survivor-bias winner.

\begin{table*}[!t]
\centering
\caption{Global runtime summary. Counts and medians use successful runs only; static rows may reflect survivor bias, and case-wise best rows are ex-post upper bounds. E2E time includes DCOPF presolve only for DC-seeded methods.}
\vspace{-1em}
\label{tab:global_summary}
\small
\setlength{\tabcolsep}{6pt}
\begin{adjustbox}{width=\textwidth}
\begin{tabular}{p{8.1cm}rrrrrrr}
\toprule
Method & Cases & Converged & Median solve [s] & Median total [s] & Median E2E [s] & Median iter & Median speedup [\%] \\
\midrule
Baseline case-data start & 19 & 18 & 3.826 & 4.005 & 4.005 & 39 & 0.0 \\
\midrule
\multicolumn{8}{l}{\textit{Static diagnostic rows (one fixed policy applied to all cases; minimum-median selection shown for illustration, not as the recommended deployment rule)}} \\
Best static Oracle AC primal-only ($V_m{+}Q_g$) & 19 & 16 & 1.837 & 1.975 & 1.975 & 37 & 2.9 \\
Best static Oracle AC primal+dual ($V_m$) & 19 & 10 & 0.270 & 0.380 & 0.380 & 54 & -93.9 \\
Best static Oracle AC constraint-dual-only ($V_a{+}V_m{+}Q_g$) & 19 & 12 & 0.495 & 0.519 & 0.519 & 73 & -51.6 \\
Best static Oracle AC bounds-dual-only ($V_m$) & 19 & 11 & 0.301 & 0.332 & 0.332 & 99 & -79.3 \\
Best static Oracle AC primal+dual (all bounds) ($V_m{+}P_g{+}Q_g$) & 19 & 14 & 0.765 & 0.812 & 0.812 & 27 & -19.8 \\
Best static DC-seeded ($V_a$) & 19 & 17 & 2.630 & 3.050 & 3.239 & 38 & 3.3 \\
\midrule
\multicolumn{8}{l}{\textit{Case-wise best (ex-post upper bound per family; descriptive only)}} \\
Case-wise best oracle AC primal-only & 19 & 18 & 3.258 & ---\rlap{$^\dagger$} & ---\rlap{$^\dagger$} & 30 & 21.6 \\
Case-wise best oracle AC primal+dual & 19 & 18 & 1.804 & ---\rlap{$^\dagger$} & ---\rlap{$^\dagger$} & 4 & 47.6 \\
Case-wise best constraint-dual & 19 & 18 & 1.944 & ---\rlap{$^\dagger$} & ---\rlap{$^\dagger$} & 6 & 44.3 \\
Case-wise best bounds-dual & 19 & 17 & 3.104 & ---\rlap{$^\dagger$} & ---\rlap{$^\dagger$} & 17 & 19.5 \\
Case-wise best primal+dual all-bounds & 19 & 18 & 1.784 & ---\rlap{$^\dagger$} & ---\rlap{$^\dagger$} & 4 & 50.3 \\
Case-wise best DC-seeded AC solve & 19 & 18 & 3.840 & ---\rlap{$^\dagger$} & ---\rlap{$^\dagger$} & 32 & 12.5 \\
Case-wise best DC end-to-end & 19 & 18 & 4.628 & ---\rlap{$^\dagger$} & ---\rlap{$^\dagger$} & 32 & -5.1 \\
\bottomrule
\end{tabular}
\end{adjustbox}
{\footnotesize $^\dagger$ For case-wise best envelopes, each case may select a different combination. The reported median solve time is therefore the median of the per-case best solve times within that family, not the solve time of any single fixed policy; total/E2E columns are omitted for that reason.}
\end{table*}
Table~\ref{tab:pairwise_tests} reports matched-case statistical tests. For each row, Med.~A and Med.~B are computed on that row's matched set of cases where both methods converge; they are therefore not always identical to the family-level medians reported in Table~\ref{tab:global_summary}. The Hodges--Lehmann estimator $\widehat{\Delta}_{HL}$ is computed from Walsh averages of the paired time differences (time~B $-$ time~A), so negative values indicate that method~B is systematically faster. The ``Wins~A'' and ``Wins~B'' columns count case-level pairwise time wins and need not agree in sign with $\widehat{\Delta}_{HL}$ when the distribution of paired differences is skewed. The ex-post case-wise best envelopes of both the O-PD and O-CD families are descriptively lower than the baseline ($p < 0.001$ on all 18 matched pairs). The O-BD family is not descriptively different from the baseline ($p = 0.6777$). The E2E comparison of baseline total time against the best DC E2E workflow is not significant ($p = 0.4171$).

\begin{table*}[!t]
\centering
\caption{Matched-case Wilcoxon signed-rank comparisons for the principal family-level benchmarks. Negative $\widehat{\Delta}_{HL}$ values favor method B. Rows marked [UB] are ex-post upper-bound comparisons, so the reported $p$-values are descriptive rather than confirmatory.}
\vspace{-1em}
\label{tab:pairwise_tests}
\small
\setlength{\tabcolsep}{4pt}
\begin{adjustbox}{width=\textwidth}
\begin{tabular}{p{6.7cm}rrrrrrrrr}
\toprule
Comparison & $n$ & Med.\ A [s] & Med.\ B [s] & Wins A & Wins B & $\widehat{\Delta}_{HL}$ [s] & Raw $p$ & Holm $p$ & Bonf.\ $p$ \\
\midrule
Baseline vs. oracle primal-only [UB] & 18 & 3.826 & 3.258 & 0 & 18 & -2.409 & $<$0.0001 & $<$0.0001 & $<$0.0001 \\
Baseline vs. oracle primal+dual [UB] & 18 & 3.826 & 1.804 & 0 & 18 & -5.329 & $<$0.0001 & $<$0.0001 & $<$0.0001 \\
Baseline solve vs. best DC solve [UB] & 18 & 3.826 & 3.840 & 4 & 14 & -1.221 & 0.0237 & 0.1184 & 0.2841 \\
Baseline total vs. best DC E2E [UB] & 18 & 4.005 & 4.628 & 12 & 6 & 0.090 & 0.4171 & 0.9114 & 1.0000 \\
Oracle primal+dual vs. DC solve [UB] & 18 & 1.804 & 3.840 & 18 & 0 & 3.570 & $<$0.0001 & $<$0.0001 & $<$0.0001 \\
Oracle primal+dual vs. DC E2E [UB] & 18 & 1.804 & 4.628 & 18 & 0 & 6.274 & $<$0.0001 & $<$0.0001 & $<$0.0001 \\
Oracle primal+dual vs. constraint-dual [UB] & 18 & 1.804 & 1.944 & 13 & 5 & 0.082 & 0.0814 & 0.3257 & 0.9771 \\
Oracle primal+dual vs. bounds-dual [UB] & 17 & 1.337 & 3.104 & 16 & 1 & 1.469 & $<$0.0001 & 0.0005 & 0.0009 \\
Oracle primal+dual vs. primal+dual all-bounds [UB] & 18 & 1.804 & 1.784 & 7 & 11 & -0.008 & 0.3038 & 0.9114 & 1.0000 \\
Constraint-dual vs. bounds-dual [UB] & 17 & 1.499 & 3.104 & 16 & 1 & 1.646 & 0.0005 & 0.0030 & 0.0060 \\
Baseline vs. constraint-dual [UB] & 18 & 3.826 & 1.944 & 0 & 18 & -5.128 & $<$0.0001 & $<$0.0001 & $<$0.0001 \\
Baseline vs. bounds-dual [UB] & 17 & 2.719 & 3.104 & 6 & 11 & -0.016 & 0.6777 & 0.9114 & 1.0000 \\
\bottomrule
\end{tabular}
\end{adjustbox}
\end{table*}

Additional case-wise scatter plots and performance-profile views are reported in \appSuppSummaries{}.

\section{Which Initialization Blocks Matter?}
\label{sec:blocks}

\subsection{Combination Heatmap}

Figure~\ref{fig:oracle_heatmap} plots median speedup across all 15 non-empty subsets of $\{P_g, Q_g, V_m, V_a\}$ for the O-PO and O-PD families.

The primal-only panel (left) shows a gradual pattern: no combination has a large negative median speedup, and multi-block combinations containing both voltage blocks and at least one generator block tend to perform best. The full-vector primal-only restart reaches $+14.3\%$ as a fixed policy, while the case-wise best O-PO envelope reaches $21.6\%$ (Table~\ref{tab:global_summary}). Because IPOPT initializes multipliers internally in this family, partial primal replacement does not introduce an explicit reused-dual inconsistency.

The block-matched primal-plus-dual panel (right) shows a sharply different pattern. Twelve of the 14 partial combinations have negative median speedups. The worst median slowdown is $P_g$ alone at $-173.1\%$, followed by $V_m{+}V_a$ at $-126.3\%$, $V_m$ at $-93.9\%$, $Q_g{+}V_m$ at $-71.1\%$, and $Q_g{+}V_a$ at $-66.5\%$. Only two partial combinations are mildly positive in median terms: $P_g{+}V_m{+}V_a$ at $+6.1\%$ and $P_g{+}Q_g{+}V_m$ at $+3.2\%$. The full-vector restart remains distinct from the rest of the family at $+47.6\%$.

The convergence picture reinforces this distinction. Several of the most harmful O-PD combinations converge on only 10--12 of 18 baseline-convergent cases and fail repeatedly on medium and large networks. The two mildly positive partial combinations also have noticeably lower convergence rates than the full-vector restart. These patterns are consistent with the residual-mismatch mechanism discussed in \appProofs{}.

\begin{figure}[!t]
\centering
\IfFileExists{figs/fig_04_oracle_combo_heatmap_speedup.pdf}{%
  \includegraphics[width=0.98\columnwidth]{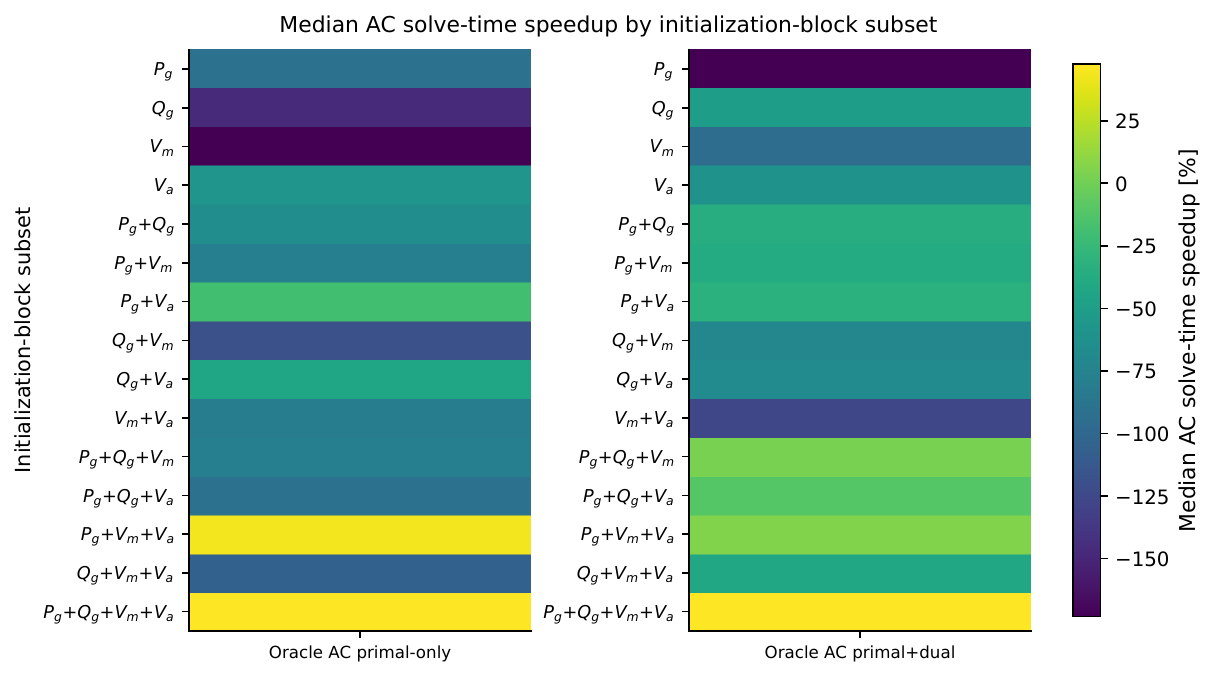}%
}{%
  \IfFileExists{figs/fig_04_oracle_combo_heatmap_speedup.png}{%
    \includegraphics[width=0.98\columnwidth]{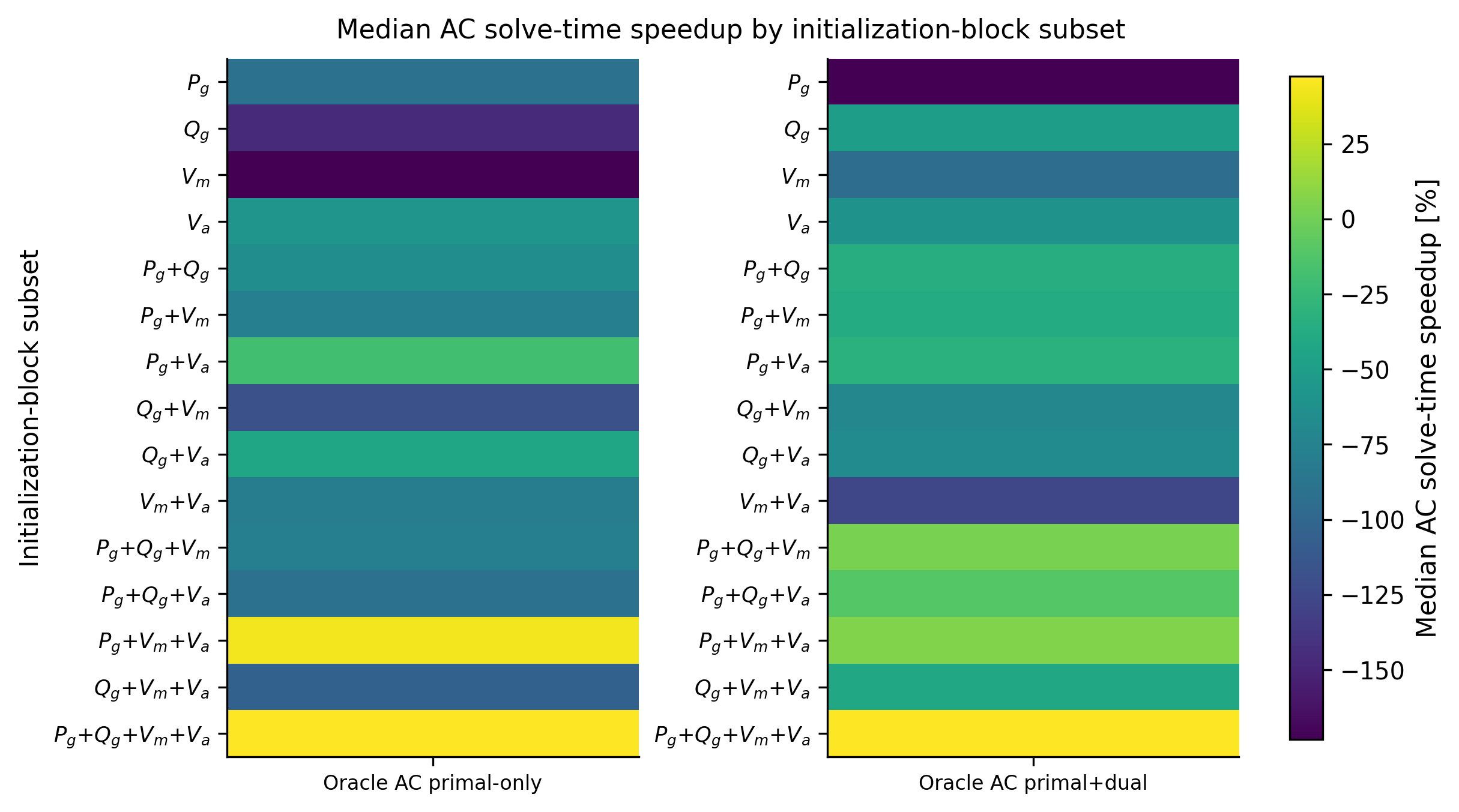}%
  }
}
\vspace{-1em}

\caption{Median AC solve-time speedup (\%) across all 15 non-empty initialization-block subsets for oracle primal-only (O-PO, left) and block-matched oracle primal-plus-dual (O-PD, right).}
\vspace{-1em}

\label{fig:oracle_heatmap}
\end{figure}

\subsection{Marginal Block Effects}

\appBlockDiagnostics{} reports marginal block summaries. While useful as secondary diagnostics, they conflate interaction effects, failure-induced case-set changes, and large-case leverage, so the combination heatmap remains the primary guide. The main point is unchanged: in the O-PO family, voltage blocks are the safest early candidates, whereas in the O-PD family the apparent marginal harm of some generator-side blocks largely reflects the dominance of harmful partial combinations rather than the value of those blocks in the full-vector restart.

\subsection{Practical Guidance and Implications for Learned Warm Starts}
\label{sec:blocks_discussion}
 
\textbf{Deployment guidance.} For O-PD restarts, use the complete $P_g{+}Q_g{+}V_m{+}V_a$ restart whenever possible (all 18 baseline-convergent cases converge; 16 within-family case wins). Partial block-matched O-PD restarts should be treated cautiously: most are harmful, and the two partial combinations with positive medians remain less reliable than the full-vector restart. If the full primal vector is not available, the \mbox{O-PD-AB} mode is more tolerant of partial coverage; several \mbox{O-PD-AB} combinations with two or three supplied blocks achieve positive medians (Section~\ref{sec:dual_decomp}). For O-PO, the family tolerates partial initialization better; combinations containing both voltage blocks ($V_m$, $V_a$) alongside at least one generator block tend to perform best.
 
The oracle results suggest four ceiling-level target priorities for learned warm starts. Since these come from same-instance exact-prediction experiments, they are ceilings, not promises about out-of-sample performance.
 
First, predicting all four primal blocks is the safest and most effective target. The fixed full-vector O-PO restart reaches a $14.3\%$ median speedup; the case-wise envelope pushes that to $21.6\%$.
 
When full coverage is not feasible because of model capacity, training data, or inference cost, voltage variables are the preferred partial target. Combinations that include both $V_m$ and $V_a$ achieve median speedups of $3.9$--$14.3\%$ in the O-PO family without a single convergence failure across the benchmark set (\appRankings{}; \appBlockDiagnostics{}). No other partial subset is as consistently safe.
 
Partial or inconsistent dual prediction without a nearly complete primal estimate is unreliable in these experiments. In the O-PD family, 12 of 14 partial configurations that supply dual variables alongside an incomplete primal vector produce negative median speedups, and several fail outright on larger cases. Unless the primal estimate already covers most of the decision vector and the predicted duals are mutually consistent with it, adding dual information tends to reduce performance.
 
If dual variables are predicted, the coverage pattern matters. The O-PD-AB results show that supplying the full bound-multiplier vector, rather than only the multipliers corresponding to the initialized primal blocks, gives higher convergence rates and better median speedups. Predicting a block-matched subset of bound multipliers supplies oracle multipliers for some bounds while leaving the rest at their default values, a combination inconsistent with the supplied primal point that perturbs the barrier complementarity conditions; the convergence penalty is visible across the benchmark.

\section{Dual Decomposition Analysis}
\label{sec:dual_decomp}

The four dual modes defined in Section~\ref{sec:oracle_pd} separate, imperfectly, the residual mechanisms discussed in \appProofs{}. Table~\ref{tab:dual_decomposition_summary} provides the aggregate view; Figure~\ref{fig:dual_heatmap} shows the combination-level detail. Because partial primal replacement also perturbs feasibility, the separation between stationarity and complementarity effects is only approximate.

\subsection{Aggregate Convergence and Speedup by Dual Mode}

The O-PO family (no duals) achieves the highest convergence rate ($98.5\%$ of 270 total runs across all 15 combinations and 18 cases) with a modest $+5.8\%$ median speedup. Introducing block-matched bound multipliers alongside constraint duals (O-PD) drops convergence to $70.4\%$ with a $-31.1\%$ median speedup. The O-PD-AB variant recovers substantially: $90.7\%$ convergence with a $+26.8\%$ median speedup, approaching the O-PO convergence rate while improving performance relative to O-PD.

The O-CD family achieves $80.4\%$ convergence with a $-13.6\%$ median speedup, while the O-BD family achieves $69.6\%$ convergence with a $-45.6\%$ median speedup. This ordering suggests that, in these experiments, restarts that reuse block-matched bound multipliers are more fragile than the constraint-dual-only restarts, although this comparison is not a clean causal isolation because the dual coverage differs across modes. The matched-case comparison is consistent with that interpretation: the case-wise best O-CD envelope is faster than the case-wise best O-BD envelope ($p = 0.0005$, Holm-corrected $p = 0.0030$; Table~\ref{tab:pairwise_tests}).

\begin{remark}[Complementarity mismatch under partial dual coverage]
\label{rem:dual_decomp}
The empirical comparison between O-BD and O-CD gives evidence on the relative importance of complementarity and stationarity channels, although partial primal replacement also affects feasibility and other residual components, so the decomposition is not clean. Bound multipliers interact directly with the barrier system through the complementarity conditions $(\tilde{x}_i - l_i)\,z_{L,i}^\star \approx \mu$ and $(u_i-\tilde{x}_i)\,z_{U,i}^\star \approx \mu$. When the supplied primal point is far from $x_i^\star$, these products deviate from the target barrier level and the solver must spend iterations re-centering. Constraint multipliers $(\lambda^\star, \mu_c^\star)$ primarily perturb stationarity and feasibility. The lower convergence rate and worse median speedup of O-BD relative to O-CD are therefore consistent with complementarity mismatch being one important source of harm, though not the only one.
\end{remark}

\subsection{Combination-Level Dual Decomposition Heatmap}

Figure~\ref{fig:dual_heatmap} extends the heatmap analysis to all five dual modes. The O-PO panel (leftmost) shows no harmful patterns. The \mbox{O-PD} panel concentrates most of the strongest negative medians in the study. The O-CD panel is less severe, with several partial combinations near break-even and the full vector achieving $+44.3\%$ in the case-wise best view. The O-BD panel has the largest median slowdowns among the decomposed dual modes, while the O-PD-AB panel is less failure-prone: many two-, three-, and four-block combinations have positive median speedups once the full bound-multiplier vector is supplied.

\begin{figure*}[!t]
\centering
\IfFileExists{figs/fig_13_dual_decomposition_heatmap.pdf}{%
  \includegraphics[width=0.98\textwidth]{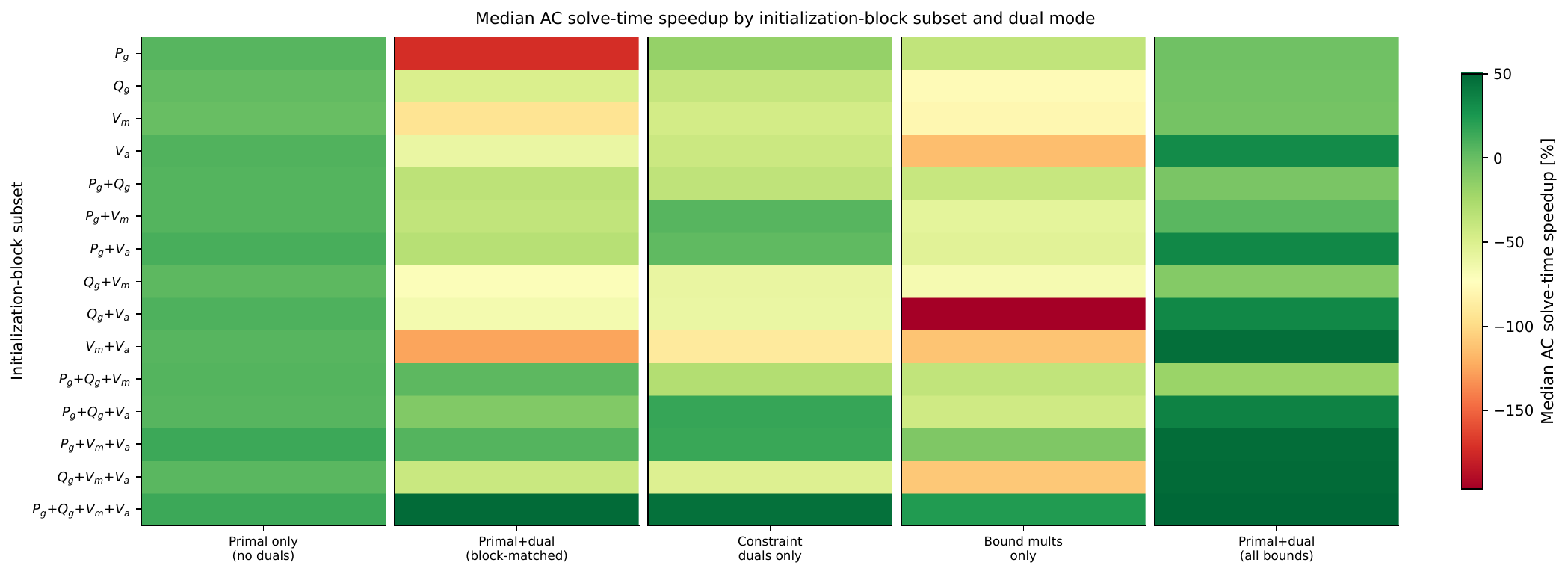}%
}{%
  \IfFileExists{figs/fig_13_dual_decomposition_heatmap.png}{%
    \includegraphics[width=0.98\textwidth]{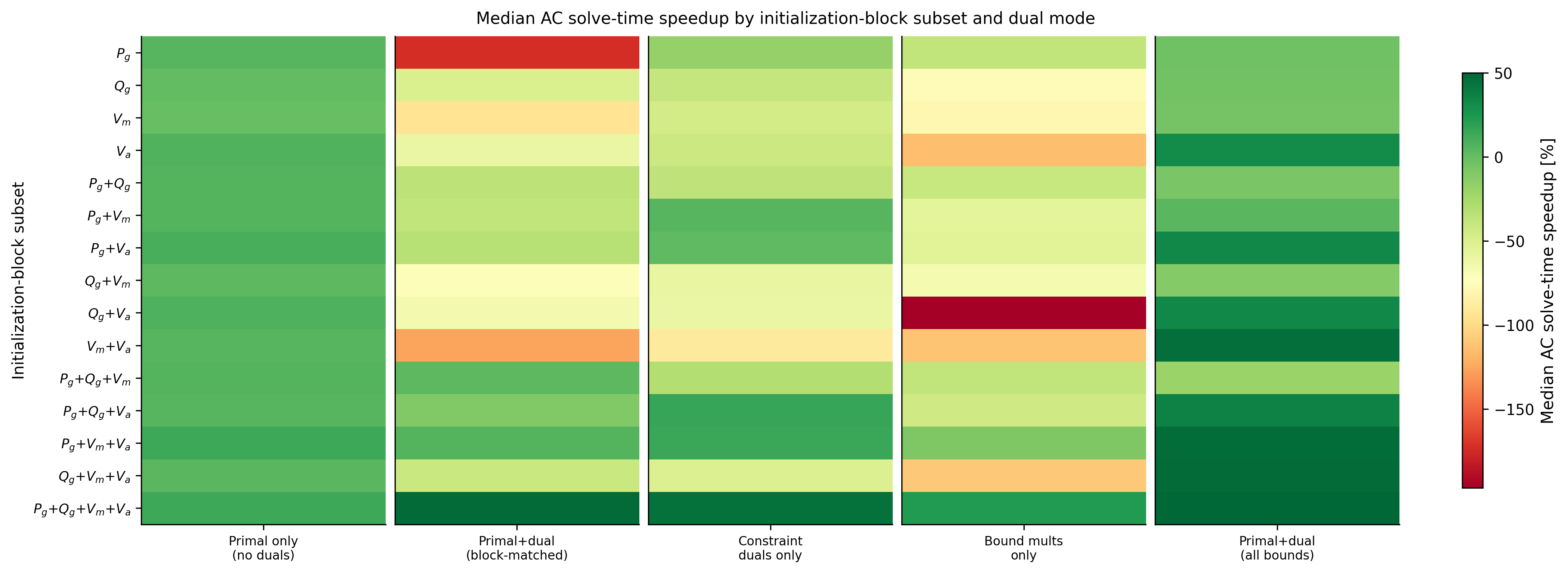}%
  }
}
\vspace{-1em}
\caption{Median AC solve-time speedup (\%) across all 15 initialization-block subsets for the five dual modes: O-PO, O-PD, O-CD, O-BD, and O-PD-AB.}
\label{fig:dual_heatmap}
\vspace{-1em}
\end{figure*}

\subsection{Dual-Only Control Experiments}

The dual-only ablations lead to the same conclusion. Constraint duals alone produce a $-39.2\%$ median slowdown (13/18 converged), bound multipliers alone produce a $-36.2\%$ median slowdown (14/18 converged), and the full dual vector without a better primal point yields only a small $+8.6\%$ median speedup with the same limited convergence rate. The bar chart is moved to \appDualDCScaling{}; Table~\ref{tab:dual_decomposition_summary} keeps the aggregate values in the main text.

\begin{table*}[!t]
\centering
\caption{Dual-mode family summary: convergence and median performance for the primal-only, primal-plus-dual, decomposed-dual, and dual-only benchmark families.}
\vspace{-1em}

\label{tab:dual_decomposition_summary}
\small
\setlength{\tabcolsep}{5pt}
\begin{adjustbox}{width=\textwidth}
\begin{tabular}{llrrrrr}
\toprule
Family & Dual mode & Total runs & Converged & Conv.\ rate [\%] & Median solve [s] & Median speedup [\%] \\
\midrule
Oracle AC primal-only & \texttt{none} & 270 & 266 & 98.5 & 3.896 & 5.8 \\
Oracle AC primal+dual & \texttt{full} & 270 & 190 & 70.4 & 0.766 & -31.1 \\
\midrule
Oracle AC constraint-dual-only & \texttt{constraint\_only} & 270 & 217 & 80.4 & 1.492 & -13.6 \\
Oracle AC bounds-dual-only & \texttt{bounds\_only} & 270 & 188 & 69.6 & 0.746 & -45.6 \\
Oracle AC primal+dual (all bounds) & \texttt{full\_all\_bounds} & 270 & 245 & 90.7 & 1.555 & 26.8 \\
\midrule
Oracle dual-only (constraint) & \texttt{constraint\_only} & 18 & 13 & 72.2 & 1.023 & -39.2 \\
Oracle dual-only (bounds) & \texttt{all\_bounds\_only} & 18 & 14 & 77.8 & 1.795 & -36.2 \\
Oracle dual-only (full) & \texttt{full\_all\_bounds} & 18 & 14 & 77.8 & 0.732 & 8.6 \\
\bottomrule
\end{tabular}
\end{adjustbox}
\end{table*}
\section{Practical DC-Seeded Methods}
\label{sec:dc_results}
 
The three DC-seeded combinations ($\{P_g\}$, $\{V_a\}$, $\{P_g, V_a\}$) all pay the same DCOPF presolve cost. They can reduce AC solve time, but Table~\ref{tab:pairwise_tests} shows that the one-shot end-to-end comparison against the baseline is not statistically significant after presolve cost is included. Thus, DC seeding is useful as physical information, but not a clear one-shot acceleration strategy in this benchmark. Combination-level DC rankings and E2E distributions are reported in \appDualDCScaling{}.

\section{Scalability and Large-Case Behavior}
\label{sec:large}
 
Table~\ref{tab:largest_cases} shows that the full-vector O-PD restart remains effective on the largest networks in the suite. On the 30,000-bus GOC instance, it cuts solve time from 165.4~s to 38.0~s, a 77\% speedup. The 6,468-bus RTE case illustrates the contrast with DC seeding: the best DC solve ($P_g{+}V_a$) takes 58.7~s versus a 45.8~s baseline, while the full-vector O-PD restart solves the same case in 8.5~s. Additional scaling plots are reported in \appDualDCScaling{}.

\begin{table*}[!t]
\centering
\caption{Warm-start performance on the eight largest baseline-convergent benchmark instances. Best oracle PD results use the full-vector $P_g{+}Q_g{+}V_m{+}V_a$ restart; ``Best DC solve blocks'' identifies the DC combination with the shortest AC solve time.}
\vspace{-1em}
\label{tab:largest_cases}
\scriptsize
\setlength{\tabcolsep}{4.0pt}
\begin{adjustbox}{width=\textwidth}
\begin{tabular}{lrrrrrrrr}
\toprule
Case & $n_b$ & $n_l$ & Baseline solve [s] & Best oracle PD [s] & Best DC solve blocks & Best DC solve [s] & Best DC E2E [s] & Best speedup (any method) [\%] \\
\midrule
\texttt{pglib\_opf\_case3022\_goc} & 3022 & 4135 & 8.975 & 4.601 & $V_a$ & 6.529 & 7.962 & 65.4 \\
\texttt{pglib\_opf\_case4020\_goc} & 4020 & 6988 & 15.534 & 4.883 & $V_a{+}P_g$ & 12.021 & 14.961 & 68.6 \\
\texttt{pglib\_opf\_case5658\_epigrids} & 5658 & 9078 & 14.982 & 8.028 & $V_a{+}P_g$ & 11.942 & 18.200 & 56.1 \\
\texttt{pglib\_opf\_case6468\_rte} & 6468 & 9000 & 45.822 & 8.511 & $V_a{+}P_g$ & 58.667\rlap{$^*$} & 61.965 & 85.7 \\
\texttt{pglib\_opf\_case7336\_epigrids} & 7336 & 11521 & 16.312 & 8.317 & $V_a{+}P_g$ & 15.204 & 20.844 & 49.0 \\
\texttt{pglib\_opf\_case10000\_goc} & 10000 & 13193 & 30.326 & 11.308 & $V_a$ & 27.365 & 34.370 & 63.1 \\
\texttt{pglib\_opf\_case13659\_pegase} & 13659 & 20467 & 45.014 & 17.136 & $V_a{+}P_g$ & 37.476 & 49.377 & 61.9 \\
\texttt{pglib\_opf\_case30000\_goc} & 30000 & 35393 & 165.383 & 38.037 & $V_a$ & 107.184 & 129.390 & 77.0 \\
\bottomrule
\multicolumn{9}{l}{\footnotesize $^*$ All three DC combinations are slower than the baseline on this case.}
\end{tabular}
\end{adjustbox}
\end{table*}

\section{Discussion}
\label{sec:discussion}

Our results show that the effect of a warm start depends strongly on which variables are initialized. Complete restarts can produce large speedups, but partial restarts often require additional iterations to repair feasibility and stationarity mismatches. That distinction explains why the same solver and the same problem class can exhibit both large speedups and large slowdowns under different warm-start choices.

The comparison between O-PD and O-PD-AB clarifies the role of bound-multiplier coverage. O-PD-AB has higher convergence rates for partial combinations. Several O-PD-AB restarts with two or three supplied blocks achieve positive median speedups, whereas the corresponding O-PD combinations have lower convergence rates and worse medians. These results suggest that giving IPOPT the full bound-multiplier vector is less disruptive than mixing oracle multipliers on some variables with default multipliers on others. That does not eliminate primal-dual inconsistency when omitted primal blocks remain at baseline, but it avoids mixing oracle bound multipliers on some variables with default bound-multiplier initialization on others.
 
For DC seeding, AC solve time and end-to-end time lead to different conclusions. DC information does sometimes help the AC solve. Whether it offsets the presolve cost in a one-shot workflow is a separate question, and the present evidence does not show a statistically significant end-to-end improvement. The raw matched-case solve-time comparison is only mildly favorable and does not survive multiple-testing correction, while the end-to-end comparison is not significant. In applications such as rolling-horizon operation, contingency analysis, or other repeated-solve settings, the relevant condition is whether the cumulative AC time saved over $k$ downstream solves exceeds the one-time DC presolve cost. That breakeven calculation lies outside the present one-shot benchmark, but it is the relevant comparison for deployment.

This study has several limitations. It uses one MATPOWER-style Pyomo/IPOPT/MUMPS stack on one Apple M4 platform, so timings and rankings may change with the formulation, hardware, linear solver, and LP/QP solver used for DCOPF presolve. The oracle experiments are same-instance second solves in one OS process; rebuilding the Pyomo model reduces model-level caching but cannot rule out smaller system-level second-solve advantages, so the reported oracle gains may slightly overstate the pure initialization effect. Timings are single-run wall-clock values; the dual decomposition is approximate because partial primal replacement also perturbs feasibility; and the DC study uses raw DCOPF outputs without AC lift, projection, or learned completion.

\section{Conclusion}
\label{sec:conclusion}

This paper benchmarked ACOPF warm starts by initialization content. Within the block-matched O-PD family, the full $P_g+Q_g+V_m+V_a$ restart is the only configuration that converges on all 18 eligible cases. Most partial block-matched primal-plus-dual restarts are harmful in median terms and fragile on large systems. The dual decomposition shows that partial bound-multiplier coverage is associated with lower robustness, whereas supplying the full bound-multiplier vector performs more reliably. DC seeding can help AC solve time, but its one-shot end-to-end benefit is not significant after presolve cost. For learned warm starts, the results suggest the following prediction priority: predict the full primal vector, emphasize voltage variables when partial coverage is necessary, and avoid partial or inconsistent dual prediction without a nearly complete primal estimate. These results define benchmark upper bounds; repeated-run timing and controlled re-solve experiments would refine the estimates.

\bibliographystyle{IEEEtran}
\bibliography{references}

\clearpage
\begingroup
\newgeometry{left=0.45cm,right=0.45cm,top=0.8cm,bottom=0.8cm}
\setlength{\tabcolsep}{1.2pt}
\renewcommand{\arraystretch}{0.82}
\fontsize{5.0}{5.4}\selectfont
\begin{landscape}
\begin{table}[p]
\centering
\caption{Complete warm-start results matrix. Each successful oracle cell reports AC solve time [s]\,/\,speedup [\%]\,/\,IPOPT iterations. DC rows follow the same convention. \textbf{F} = solver failure. Column headers: \textbf{5} = \texttt{pglib\_opf\_case5\_pjm}; \textbf{14} = \texttt{pglib\_opf\_case14\_ieee}; \textbf{39} = \texttt{pglib\_opf\_case39\_epri}; \textbf{57} = \texttt{pglib\_opf\_case57\_ieee}; \textbf{118} = \texttt{pglib\_opf\_case118\_ieee}; \textbf{200} = \texttt{pglib\_opf\_case200\_activ}; \textbf{300} = \texttt{pglib\_opf\_case300\_ieee}; \textbf{500} = \texttt{pglib\_opf\_case500\_ieee}; \textbf{1,354} = \texttt{pglib\_opf\_case1354\_pegase}; \textbf{2,000} = \texttt{pglib\_opf\_case2000\_goc}; \textbf{3,022} = \texttt{pglib\_opf\_case3022\_goc}; \textbf{4,020} = \texttt{pglib\_opf\_case4020\_goc}; \textbf{5,658} = \texttt{pglib\_opf\_case5658\_epigrids}; \textbf{6,468} = \texttt{pglib\_opf\_case6468\_rte}; \textbf{7,336} = \texttt{pglib\_opf\_case7336\_epigrids}; \textbf{8,387} = \texttt{pglib\_opf\_case8387\_pegase}; \textbf{10,000} = \texttt{pglib\_opf\_case10000\_goc}; \textbf{13,659} = \texttt{pglib\_opf\_case13659\_pegase}; \textbf{30,000} = \texttt{pglib\_opf\_case30000\_goc}.}
\label{tab:all_warmstarts_matrix}
\begin{adjustbox}{max width=\linewidth,max totalheight=0.90\textheight,keepaspectratio}
\begin{tabular}{p{1.55cm}p{2.00cm}rrrrrr*{19}{c}}
\toprule
Family & Combination & Cases & OK & MedSolve[s] & MedIter & MedE2E[s] & MedSpd[\%] & 5 & 14 & 39 & 57 & 118 & 200 & 300 & 500 & 1K & 2K & 3,022 & 4K & 5,658 & 6K & 7,336 & 8,387 & 10,000 & 14K & 30,000 \\
\midrule
\multicolumn{27}{l}{\textit{Baseline case-data start}} \\
Baseline & --- & 19 & 18 & 3.826 & 39 & 4.005 & 0.0 & 0.023/0/20 & 0.036/0/14 & 0.061/0/26 & 0.095/0/14 & 0.174/0/23 & 0.237/0/24 & 0.489/0/31 & 0.852/0/33 & 2.719/0/38 & 4.932/0/43 & 8.975/0/53 & 15.534/0/57 & 14.982/0/40 & 45.822/0/149 & 16.312/0/41 & \textbf{F} & 30.326/0/79 & 45.014/0/66 & 165.383/0/188 \\
\midrule
\multicolumn{27}{l}{\textit{Oracle AC primal-only (O-PO): no duals}} \\
O-PO & $V_a$ & 18 & 18 & 4.545 & 38 & 5.124 & 7.1 & 0.021/8/15 & 0.033/9/15 & 0.062/-2/30 & 0.073/24/14 & 0.336/-93/32 & 0.234/1/22 & 0.398/19/25 & 1.315/-54/60 & 2.485/9/33 & 6.605/-34/61 & 8.384/7/50 & 15.317/1/56 & 13.341/11/37 & 35.976/21/137 & 16.530/-1/40 &  & 26.693/12/65 & 161.675/-259/264 & 111.568/33/111 \\
O-PO & $V_m$ & 18 & 17 & 2.820 & 40 & 3.778 & -0.8 & 0.022/5/20 & 0.034/3/14 & 0.063/-4/27 & 0.070/27/11 & 1.121/-544/23 & 0.233/2/22 & 0.558/-14/51 & 0.859/-1/33 & 2.820/-4/40 & 4.711/4/40 & 13.944/-55/112 & 15.072/3/55 & 14.196/5/36 & 37.294/19/120 & 16.586/-2/40 &  & \textbf{F} & 126.866/-182/187 & 204.837/-24/212 \\
O-PO & $P_g$ & 18 & 18 & 4.369 & 37 & 4.942 & 4.9 & 0.022/5/18 & 0.031/13/14 & 0.055/10/23 & 0.068/29/11 & 0.371/-113/21 & 0.220/7/19 & 0.460/6/31 & 0.827/3/31 & 3.139/-15/42 & 5.600/-14/53 & 8.196/9/54 & 9.810/37/36 & 14.491/3/38 & 50.481/-10/208 & 15.609/4/39 &  & 29.012/4/77 & 41.202/8/60 & 175.244/-6/190 \\
O-PO & $Q_g$ & 18 & 18 & 4.306 & 40 & 5.260 & 1.0 & 0.020/14/20 & 0.030/16/14 & 0.062/-2/32 & 0.085/11/14 & 0.412/-136/24 & 0.236/1/23 & 0.425/13/29 & 0.852/0/33 & 3.724/-37/39 & 4.889/1/44 & 7.219/20/54 & 13.627/12/57 & 14.394/4/40 & 73.242/-60/276 & 16.119/1/41 &  & 31.480/-4/82 & 717.549/-1494/262 & 148.672/10/166 \\
O-PO & $V_a{+}V_m$ & 18 & 18 & 4.638 & 36 & 5.223 & 5.6 & 0.019/16/14 & 0.034/5/13 & 0.063/-4/29 & 0.075/21/13 & 0.186/-7/27 & 0.235/1/20 & 0.405/17/25 & 1.776/-108/36 & 4.068/-50/36 & 5.208/-6/44 & 6.764/25/50 & 13.267/15/55 & 14.963/0/37 & 15.574/66/47 & 15.792/3/37 &  & 27.230/10/62 & 42.269/6/61 & 100.819/39/104 \\
O-PO & $V_a{+}P_g$ & 18 & 18 & 4.282 & 32 & 4.881 & 9.7 & 0.021/7/15 & 0.029/19/13 & 0.069/-14/29 & 0.072/24/14 & 0.194/-12/26 & 0.215/9/16 & 0.392/20/25 & 1.003/-18/39 & 2.679/1/36 & 5.884/-19/52 & 8.074/10/43 & 9.274/40/31 & 11.639/22/32 & 33.451/27/120 & 14.771/9/33 &  & 56.554/-86/144 & 35.492/21/48 & 118.955/28/121 \\
O-PO & $V_a{+}Q_g$ & 18 & 18 & 4.561 & 38 & 5.132 & 8.1 & 0.017/25/14 & 0.028/20/14 & 0.059/3/26 & 0.079/17/13 & 0.193/-11/27 & 0.241/-2/23 & 0.407/17/25 & 2.129/-150/60 & 2.544/6/33 & 6.578/-33/61 & 8.093/10/49 & 13.956/10/55 & 12.623/16/37 & 57.152/-25/224 & 16.544/-1/40 &  & 27.305/10/63 & 143.383/-219/239 & 117.216/29/119 \\
O-PO & $V_m{+}P_g$ & 18 & 18 & 4.441 & 34 & 4.754 & 5.7 & 0.018/22/18 & 0.034/4/15 & 0.057/6/27 & 0.071/26/11 & 0.200/-15/22 & 0.223/6/19 & 0.461/6/28 & 0.800/6/29 & 4.493/-65/41 & 4.390/11/36 & 11.333/-26/83 & 9.460/39/33 & 12.201/19/34 & 71.894/-57/239 & 15.832/3/37 &  & 61.833/-104/169 & 42.100/6/60 & 176.229/-7/186 \\
O-PO & $V_m{+}Q_g$ & 18 & 16 & 1.837 & 37 & 1.975 & 2.9 & 0.021/7/19 & 0.029/17/15 & 0.062/-2/27 & 0.070/27/12 & 0.190/-9/23 & 0.237/0/23 & 0.544/-11/46 & 0.890/-4/35 & 2.785/-2/40 & 4.521/8/38 & 8.495/5/73 & 15.075/3/55 & 12.649/16/36 & 44.475/3/144 & 16.609/-2/40 &  & \textbf{F} & \textbf{F} & 130.996/21/137 \\
O-PO & $P_g{+}Q_g$ & 18 & 17 & 2.948 & 36 & 3.334 & 6.6 & 0.023/-1/18 & 0.028/21/15 & 0.056/7/23 & 0.081/15/13 & 0.280/-61/22 & 0.233/2/19 & 0.443/9/30 & 0.825/3/31 & 2.948/-8/42 & 5.110/-4/46 & 6.690/25/54 & 11.518/26/36 & 12.408/17/37 & \textbf{F} & 15.096/7/37 &  & 29.719/2/78 & 42.037/7/61 & 220.825/-34/246 \\
O-PO & $V_a{+}V_m{+}P_g$ & 18 & 18 & 3.572 & 30 & 4.165 & 14.0 & 0.022/5/13 & 0.030/15/13 & 0.049/19/18 & 0.316/-232/10 & 0.172/1/22 & 0.205/13/16 & 0.386/21/22 & 0.815/4/28 & 2.439/10/31 & 4.704/5/37 & 5.839/35/43 & 10.705/31/30 & 13.742/8/32 & 13.525/70/32 & 14.135/13/30 &  & 23.929/21/50 & 36.539/19/49 & 119.943/27/122 \\
O-PO & $V_a{+}V_m{+}Q_g$ & 18 & 18 & 3.806 & 36 & 4.286 & 3.9 & 0.018/20/13 & 0.031/14/13 & 0.060/1/28 & 0.185/-95/13 & 0.192/-10/28 & 0.224/6/20 & 0.412/16/21 & 0.908/-6/34 & 2.625/3/36 & 4.988/-1/40 & 8.096/10/50 & 15.641/-1/55 & 14.962/0/37 & 17.127/63/50 & 15.950/2/37 &  & 28.013/8/64 & 43.072/4/62 & 100.282/39/103 \\
O-PO & $V_a{+}P_g{+}Q_g$ & 18 & 18 & 4.151 & 32 & 4.626 & 4.9 & 0.022/3/15 & 0.037/-3/14 & 0.058/5/24 & 0.112/-17/14 & 0.189/-8/26 & 0.225/5/17 & 0.382/22/24 & 1.133/-33/48 & 2.587/5/35 & 5.715/-16/50 & 7.511/16/43 & 10.525/32/30 & 13.766/8/32 & 58.941/-29/244 & 14.661/10/32 &  & 61.225/-102/154 & 34.949/22/47 & 120.029/27/122 \\
O-PO & $V_m{+}P_g{+}Q_g$ & 18 & 18 & 3.735 & 34 & 4.393 & 5.7 & 0.019/16/19 & 0.034/4/15 & 0.059/4/27 & 0.078/18/11 & 0.174/0/21 & 0.219/8/19 & 0.411/16/27 & 0.925/-9/27 & 3.013/-11/42 & 4.457/10/37 & 8.785/2/75 & 9.545/39/33 & 13.316/11/34 & 55.425/-21/219 & 15.717/4/36 &  & 37.202/-23/93 & 41.575/8/59 & 142.885/14/153 \\
O-PO & $V_a{+}V_m{+}P_g{+}Q_g$ & 18 & 18 & 3.872 & 30 & 4.445 & 14.3 & 0.020/13/13 & 0.029/18/13 & 0.054/12/17 & 0.290/-205/11 & 0.160/8/18 & 0.205/14/15 & 0.428/12/15 & 0.804/6/27 & 3.666/-35/31 & 4.078/17/30 & 5.861/35/44 & 10.865/30/30 & 13.404/11/30 & 11.728/74/33 & 13.903/15/29 &  & 23.843/21/50 & 35.682/21/47 & 118.407/28/120 \\
\midrule
\multicolumn{27}{l}{\textit{Oracle AC primal+dual (O-PD): constraint duals + block-matched bound mults}} \\
O-PD & $V_a$ & 18 & 12 & 1.298 & 116 & 1.336 & -60.6 & 0.018/22/14 & 0.035/1/19 & 0.101/-66/80 & 0.107/-12/22 & 0.270/-55/63 & 0.296/-25/37 & 2.300/-371/242 & 2.883/-238/187 & 10.583/-289/224 & 13.946/-183/152 & \textbf{F} & \textbf{F} & 65.265/-336/268 & 60.007/-31/254 & \textbf{F} &  & \textbf{F} & \textbf{F} & \textbf{F} \\
O-PD & $V_m$ & 18 & 10 & 0.270 & 54 & 0.380 & -93.9 & 0.026/-15/39 & 0.032/10/10 & 0.078/-29/49 & 0.249/-161/55 & 0.252/-45/53 & 0.289/-22/38 & 4.080/-735/286 & 4.616/-442/185 & \textbf{F} & 12.001/-143/133 & \textbf{F} & \textbf{F} & 59.273/-296/249 & \textbf{F} & \textbf{F} &  & \textbf{F} & \textbf{F} & \textbf{F} \\
O-PD & $P_g$ & 18 & 13 & 1.260 & 58 & 1.361 & -173.1 & 0.022/2/23 & 0.031/13/16 & 0.080/-31/56 & 0.379/-298/44 & 0.232/-33/45 & 0.216/9/16 & 1.762/-261/226 & 1.260/-48/58 & 9.109/-235/179 & 16.890/-242/195 & \textbf{F} & 53.503/-244/180 & \textbf{F} & \textbf{F} & 52.879/-224/169 &  & \textbf{F} & 122.935/-173/200 & \textbf{F} \\
O-PD & $Q_g$ & 18 & 10 & 0.301 & 60 & 0.327 & -49.5 & 0.024/-4/34 & 0.039/-9/32 & 0.072/-19/49 & 0.112/-18/25 & 0.298/-71/72 & 0.304/-28/40 & 1.455/-198/187 & \textbf{F} & 15.650/-476/268 & 15.495/-214/190 & \textbf{F} & \textbf{F} & 63.332/-323/251 & \textbf{F} & \textbf{F} &  & \textbf{F} & \textbf{F} & \textbf{F} \\
O-PD & $V_a{+}V_m$ & 18 & 11 & 0.567 & 108 & 0.583 & -126.3 & 0.028/-22/29 & 0.033/7/9 & 0.070/-15/45 & 0.148/-55/33 & 0.292/-68/70 & 0.567/-139/108 & 1.106/-126/127 & 3.602/-323/222 & 14.086/-418/283 & 14.372/-191/150 & \textbf{F} & \textbf{F} & 61.359/-310/256 & \textbf{F} & \textbf{F} &  & \textbf{F} & \textbf{F} & \textbf{F} \\
O-PD & $V_a{+}P_g$ & 18 & 15 & 1.369 & 57 & 1.413 & -32.2 & 0.018/19/9 & 0.030/16/13 & 0.084/-38/57 & 0.201/-111/14 & 0.199/-14/33 & 0.216/9/16 & 1.369/-180/183 & 0.993/-17/41 & 5.104/-88/93 & 17.592/-257/190 & 7.487/17/65 & 20.537/-32/71 & 15.160/-1/37 & \textbf{F} & 26.425/-62/71 &  & \textbf{F} & 71.078/-58/103 & \textbf{F} \\
O-PD & $V_a{+}Q_g$ & 18 & 11 & 0.292 & 87 & 0.326 & -66.5 & 0.024/-3/26 & 0.032/11/14 & 0.098/-62/87 & 0.125/-32/30 & 0.290/-66/73 & 0.292/-23/38 & 2.355/-382/260 & 4.441/-421/175 & 15.731/-479/299 & 14.382/-192/180 & \textbf{F} & \textbf{F} & 61.633/-311/257 & \textbf{F} & \textbf{F} &  & \textbf{F} & \textbf{F} & \textbf{F} \\
O-PD & $V_m{+}P_g$ & 18 & 14 & 0.822 & 46 & 0.888 & -37.0 & 0.031/-34/49 & 0.029/20/8 & 0.068/-12/40 & 0.137/-44/13 & 0.198/-14/35 & 0.216/9/16 & 0.684/-40/73 & 0.960/-13/42 & 13.459/-395/248 & 7.461/-51/73 & 28.360/-216/273 & 67.218/-333/253 & 14.513/3/34 & \textbf{F} & 32.362/-98/94 &  & \textbf{F} & \textbf{F} & \textbf{F} \\
O-PD & $V_m{+}Q_g$ & 18 & 10 & 0.271 & 66 & 0.288 & -71.1 & 0.022/5/26 & 0.035/1/21 & 0.101/-66/82 & 0.168/-76/32 & 0.243/-40/51 & 0.300/-26/37 & 1.132/-132/135 & 3.011/-253/209 & \textbf{F} & 13.951/-183/171 & \textbf{F} & \textbf{F} & 55.697/-272/239 & \textbf{F} & \textbf{F} &  & \textbf{F} & \textbf{F} & \textbf{F} \\
O-PD & $P_g{+}Q_g$ & 18 & 13 & 0.614 & 58 & 0.682 & -35.6 & 0.027/-19/34 & 0.032/10/17 & 0.069/-14/39 & 0.193/-102/28 & 0.236/-36/47 & 0.209/12/15 & 0.614/-26/58 & 1.489/-75/73 & 12.180/-348/215 & 12.445/-152/141 & \textbf{F} & 49.377/-218/179 & \textbf{F} & 60.688/-32/235 & 64.296/-294/208 &  & \textbf{F} & \textbf{F} & \textbf{F} \\
O-PD & $V_a{+}V_m{+}P_g$ & 18 & 15 & 1.273 & 31 & 1.406 & 6.1 & 0.019/18/8 & 0.027/24/7 & 0.065/-6/38 & 0.074/22/7 & 0.171/2/23 & 0.223/6/16 & 1.273/-160/94 & 0.869/-2/28 & 4.259/-57/73 & 7.070/-43/66 & 6.682/26/54 & 11.900/23/31 & 12.299/18/24 & 36.131/21/113 & 23.735/-46/61 &  & \textbf{F} & \textbf{F} & \textbf{F} \\
O-PD & $V_a{+}V_m{+}Q_g$ & 18 & 10 & 0.281 & 58 & 0.291 & -40.7 & 0.020/13/23 & 0.031/13/6 & 0.077/-28/56 & 0.123/-29/34 & 0.265/-52/59 & 0.297/-25/40 & 1.008/-106/101 & 3.274/-284/223 & \textbf{F} & 19.502/-295/236 & \textbf{F} & \textbf{F} & 58.597/-291/244 & \textbf{F} & \textbf{F} &  & \textbf{F} & \textbf{F} & \textbf{F} \\
O-PD & $V_a{+}P_g{+}Q_g$ & 18 & 14 & 1.983 & 50 & 2.054 & -10.5 & 0.016/30/7 & 0.029/20/12 & 0.051/16/16 & 0.080/16/10 & 0.184/-5/27 & 0.190/20/11 & 2.932/-500/229 & 1.035/-21/47 & 7.548/-178/109 & 6.906/-40/70 & 10.365/-15/72 & 15.744/-1/57 & 19.755/-32/52 & \textbf{F} & 28.380/-74/77 &  & \textbf{F} & \textbf{F} & \textbf{F} \\
O-PD & $V_m{+}P_g{+}Q_g$ & 18 & 14 & 0.753 & 27 & 0.806 & 3.2 & 0.017/24/13 & 0.027/25/9 & 0.054/11/22 & 0.090/6/9 & 0.159/9/19 & 0.179/25/6 & 0.659/-35/70 & 0.848/1/32 & 13.303/-389/242 & 7.145/-45/64 & 30.043/-235/266 & \textbf{F} & 10.698/29/17 & \textbf{F} & 24.931/-53/65 &  & \textbf{F} & 169.398/-276/283 & \textbf{F} \\
O-PD & $V_a{+}V_m{+}P_g{+}Q_g$ & 18 & 18 & 1.804 & 4 & 2.411 & 47.6 & 0.019/19/2 & 0.027/24/2 & 0.040/34/2 & 0.071/25/2 & 0.119/32/2 & 0.167/30/2 & 0.272/44/2 & 0.486/43/3 & 1.337/51/4 & 2.270/54/3 & 4.601/49/9 & 4.883/69/4 & 8.028/46/4 & 8.511/81/4 & 8.317/49/4 &  & 11.308/63/8 & 17.136/62/5 & 38.037/77/12 \\
\midrule
\multicolumn{27}{l}{\textit{Oracle AC constraint-dual (O-CD): constraint duals only, no bound mults}} \\
O-CD & $V_a$ & 18 & 14 & 1.339 & 89 & 1.411 & -41.9 & 0.017/24/11 & 0.033/7/11 & 0.070/-15/45 & 0.088/8/28 & 0.194/-11/34 & 0.234/1/21 & 1.243/-154/148 & 1.436/-68/80 & 6.723/-147/98 & 10.001/-103/110 & \textbf{F} & 48.125/-210/202 & 28.366/-89/104 & 43.258/6/161 & 76.200/-367/228 &  & \textbf{F} & \textbf{F} & \textbf{F} \\
O-CD & $V_m$ & 18 & 14 & 1.441 & 81 & 1.514 & -45.9 & 0.023/-1/26 & 0.033/8/13 & 0.079/-30/50 & 0.086/10/20 & 0.208/-20/38 & 0.223/6/20 & 1.289/-164/164 & 1.593/-87/88 & 7.673/-182/153 & 7.334/-49/74 & \textbf{F} & 70.327/-353/287 & 31.711/-112/115 & 65.605/-43/242 & 49.123/-201/159 &  & \textbf{F} & \textbf{F} & \textbf{F} \\
O-CD & $P_g$ & 18 & 15 & 1.635 & 42 & 1.679 & -17.4 & 0.027/-17/17 & 0.034/3/14 & 0.066/-9/36 & 0.086/10/15 & 0.208/-19/36 & 0.197/17/11 & 1.298/-166/168 & 1.635/-92/42 & 9.919/-265/157 & 5.370/-9/46 & 26.825/-199/265 & 16.022/-3/55 & 13.932/7/29 & \textbf{F} & 19.238/-18/45 &  & 81.437/-169/237 & \textbf{F} & \textbf{F} \\
O-CD & $Q_g$ & 18 & 13 & 0.612 & 63 & 0.665 & -39.1 & 0.025/-8/25 & 0.029/18/14 & 0.065/-8/29 & 0.105/-10/23 & 0.242/-39/49 & 0.237/0/22 & 0.612/-25/63 & 2.782/-226/108 & 11.728/-331/196 & 7.545/-53/81 & \textbf{F} & 61.140/-294/261 & 29.645/-98/108 & \textbf{F} & 45.044/-176/143 &  & \textbf{F} & \textbf{F} & \textbf{F} \\
O-CD & $V_a{+}V_m$ & 18 & 13 & 0.925 & 90 & 0.964 & -89.3 & 0.030/-29/13 & 0.029/19/10 & 0.100/-64/83 & 0.077/20/16 & 0.256/-47/57 & 0.429/-81/70 & 0.925/-89/90 & 1.979/-132/93 & 6.473/-138/126 & 12.135/-146/119 & \textbf{F} & 68.701/-342/280 & 31.916/-113/118 & \textbf{F} & 71.936/-341/224 &  & \textbf{F} & \textbf{F} & \textbf{F} \\
O-CD & $V_a{+}P_g$ & 18 & 17 & 3.013 & 33 & 3.246 & 2.7 & 0.017/24/10 & 0.033/6/11 & 0.067/-10/36 & 0.082/14/15 & 0.163/6/18 & 0.187/21/11 & 1.837/-276/183 & 0.779/9/24 & 3.013/-11/44 & 5.253/-7/43 & 7.442/17/33 & 22.874/-47/95 & 12.739/15/21 & 63.046/-38/253 & 15.865/3/33 &  & 80.306/-165/234 & 105.372/-134/166 & \textbf{F} \\
O-CD & $V_a{+}Q_g$ & 18 & 12 & 0.909 & 72 & 0.965 & -59.9 & 0.024/-5/9 & 0.029/17/10 & 0.073/-19/50 & 0.086/9/27 & 0.284/-63/62 & 0.251/-6/26 & 1.942/-297/113 & 1.533/-80/86 & 8.924/-228/151 & 7.730/-57/83 & \textbf{F} & 43.482/-180/204 & 30.015/-100/106 & \textbf{F} & \textbf{F} &  & \textbf{F} & \textbf{F} & \textbf{F} \\
O-CD & $V_m{+}P_g$ & 18 & 14 & 0.989 & 24 & 1.057 & 5.1 & 0.019/18/20 & 0.034/4/11 & 0.056/8/22 & 0.073/23/13 & 0.163/6/20 & 0.193/19/11 & 1.242/-154/151 & 0.737/14/24 & 8.355/-207/148 & 5.866/-19/53 & 17.721/-97/159 & 20.020/-29/75 & 12.621/16/23 & \textbf{F} & 21.419/-31/52 &  & \textbf{F} & \textbf{F} & \textbf{F} \\
O-CD & $V_m{+}Q_g$ & 18 & 14 & 1.543 & 85 & 1.598 & -59.6 & 0.026/-13/25 & 0.035/2/16 & 0.088/-45/62 & 0.111/-16/41 & 0.226/-30/44 & 0.224/6/21 & 1.492/-205/168 & 1.593/-87/93 & 11.154/-310/192 & 7.257/-47/77 & \textbf{F} & 64.782/-317/276 & 31.016/-107/115 & 78.844/-72/300 & 50.658/-211/164 &  & \textbf{F} & \textbf{F} & \textbf{F} \\
O-CD & $P_g{+}Q_g$ & 18 & 14 & 1.265 & 50 & 1.325 & -36.0 & 0.024/-3/32 & 0.034/4/15 & 0.070/-16/42 & 0.082/14/19 & 0.242/-39/46 & 0.186/22/9 & 1.377/-182/180 & 1.152/-35/54 & 8.003/-194/155 & 7.237/-47/76 & 29.180/-225/295 & 21.238/-37/81 & 17.541/-17/34 & \textbf{F} & 27.433/-68/76 &  & \textbf{F} & \textbf{F} & \textbf{F} \\
O-CD & $V_a{+}V_m{+}P_g$ & 18 & 16 & 1.756 & 20 & 2.007 & 15.4 & 0.018/23/11 & 0.031/14/7 & 0.054/11/21 & 0.068/28/9 & 0.145/17/12 & 0.204/14/10 & 0.640/-31/67 & 0.642/25/15 & 2.869/-6/41 & 5.748/-17/52 & 5.610/37/41 & 8.499/45/19 & 10.915/27/17 & 21.562/53/61 & 15.125/7/29 &  & 84.262/-178/213 & \textbf{F} & \textbf{F} \\
O-CD & $V_a{+}V_m{+}Q_g$ & 18 & 12 & 0.495 & 73 & 0.519 & -51.6 & 0.019/15/8 & 0.031/13/7 & 0.091/-50/69 & 0.077/19/16 & 0.242/-39/52 & 0.243/-2/25 & 0.746/-53/77 & 1.764/-107/102 & \textbf{F} & 11.833/-140/135 & \textbf{F} & 70.760/-356/276 & 31.788/-112/117 & \textbf{F} & 63.945/-292/201 &  & \textbf{F} & \textbf{F} & \textbf{F} \\
O-CD & $V_a{+}P_g{+}Q_g$ & 18 & 15 & 1.236 & 33 & 1.280 & 15.9 & 0.017/28/6 & 0.030/16/10 & 0.048/20/16 & 0.075/21/16 & 0.171/2/23 & 0.187/21/9 & 1.236/-153/153 & 0.968/-14/33 & 4.347/-60/74 & 4.909/0/42 & 5.494/39/37 & 10.136/35/34 & 12.305/18/24 & 43.242/6/163 & 18.733/-15/43 &  & \textbf{F} & \textbf{F} & \textbf{F} \\
O-CD & $V_m{+}P_g{+}Q_g$ & 18 & 16 & 3.724 & 42 & 4.200 & -30.4 & 0.033/-44/35 & 0.031/14/15 & 0.058/5/28 & 0.079/17/14 & 0.183/-5/27 & 0.173/27/7 & 1.983/-306/262 & 1.495/-75/29 & 9.579/-252/161 & 5.465/-11/50 & 18.987/-112/182 & 19.614/-26/71 & 10.501/30/23 & 65.695/-43/242 & 21.946/-35/55 &  & 64.367/-112/184 & \textbf{F} & \textbf{F} \\
O-CD & $V_a{+}V_m{+}P_g{+}Q_g$ & 18 & 18 & 1.944 & 6 & 2.537 & 44.3 & 0.014/37/3 & 0.028/21/3 & 0.043/29/4 & 0.066/30/4 & 0.124/29/4 & 0.167/30/3 & 0.285/42/6 & 0.479/44/4 & 1.499/45/8 & 2.389/52/5 & 5.112/43/12 & 5.285/66/6 & 6.703/55/5 & 8.342/82/11 & 8.972/45/6 &  & 11.432/62/9 & 18.378/59/8 & 41.639/75/16 \\
\midrule
\multicolumn{27}{l}{\textit{Oracle AC bounds-dual (O-BD): block-matched bound mults only, no constraint duals}} \\
O-BD & $V_a$ & 18 & 11 & 0.373 & 104 & 0.397 & -114.4 & 0.034/-47/48 & 0.037/-3/24 & 0.087/-43/67 & 0.087/8/30 & 0.373/-114/104 & 0.328/-38/52 & 1.970/-303/262 & 4.527/-431/240 & 13.094/-382/253 & 17.747/-260/235 & \textbf{F} & \textbf{F} & 39.683/-165/171 & \textbf{F} & \textbf{F} &  & \textbf{F} & \textbf{F} & \textbf{F} \\
O-BD & $V_m$ & 18 & 11 & 0.301 & 99 & 0.332 & -79.3 & 0.043/-89/99 & 0.032/10/17 & 0.095/-57/79 & 0.089/6/28 & 0.277/-59/66 & 0.301/-27/42 & 0.876/-79/112 & 3.093/-263/209 & \textbf{F} & 13.033/-164/165 & \textbf{F} & \textbf{F} & 30.081/-101/127 & \textbf{F} & 67.121/-311/250 &  & \textbf{F} & \textbf{F} & \textbf{F} \\
O-BD & $P_g$ & 18 & 13 & 1.521 & 74 & 1.593 & -37.0 & 0.023/1/36 & 0.031/12/13 & 0.094/-55/79 & 0.111/-17/50 & 0.242/-39/52 & 0.229/4/18 & 1.927/-294/217 & 1.521/-78/83 & 12.585/-363/258 & 6.758/-37/74 & \textbf{F} & 18.436/-19/77 & 16.410/-10/49 & \textbf{F} & 38.776/-138/131 &  & \textbf{F} & \textbf{F} & \textbf{F} \\
O-BD & $Q_g$ & 18 & 10 & 0.323 & 83 & 0.346 & -76.3 & 0.028/-20/45 & 0.034/4/22 & 0.085/-39/68 & 0.117/-23/51 & 0.371/-113/98 & 0.274/-15/36 & 1.240/-154/118 & 3.007/-253/206 & \textbf{F} & 12.514/-154/163 & \textbf{F} & \textbf{F} & 45.060/-201/194 & \textbf{F} & \textbf{F} &  & \textbf{F} & \textbf{F} & \textbf{F} \\
O-BD & $V_a{+}V_m$ & 18 & 13 & 0.752 & 97 & 0.800 & -111.1 & 0.019/19/16 & 0.038/-6/29 & 0.095/-56/76 & 0.087/9/27 & 0.368/-111/60 & 0.536/-126/97 & 0.752/-54/88 & 4.983/-485/222 & 13.654/-402/300 & 16.636/-237/213 & \textbf{F} & 46.884/-202/247 & 31.507/-110/128 & \textbf{F} & 71.460/-338/268 &  & \textbf{F} & \textbf{F} & \textbf{F} \\
O-BD & $V_a{+}P_g$ & 18 & 14 & 1.611 & 75 & 1.670 & -54.1 & 0.028/-20/45 & 0.033/6/21 & 0.067/-11/42 & 0.084/12/24 & 0.740/-325/56 & 0.232/2/22 & 1.937/-296/272 & 1.284/-51/69 & 9.995/-268/212 & 7.774/-58/90 & 21.219/-136/214 & 17.958/-16/81 & 26.525/-77/100 & \textbf{F} & 52.350/-221/189 &  & \textbf{F} & \textbf{F} & \textbf{F} \\
O-BD & $V_a{+}Q_g$ & 18 & 11 & 0.517 & 72 & 0.552 & -196.8 & 0.026/-13/27 & 0.035/2/21 & 0.069/-14/44 & 0.086/10/25 & 0.517/-197/72 & 0.330/-39/52 & 1.602/-228/219 & 3.135/-268/217 & 11.634/-328/250 & 22.378/-354/293 & \textbf{F} & \textbf{F} & 45.190/-202/199 & \textbf{F} & \textbf{F} &  & \textbf{F} & \textbf{F} & \textbf{F} \\
O-BD & $V_m{+}P_g$ & 18 & 12 & 0.715 & 74 & 0.765 & -56.7 & 0.042/-82/99 & 0.035/3/15 & 0.123/-103/110 & 0.106/-11/45 & 0.316/-82/71 & 0.239/-1/25 & 1.173/-140/125 & 1.113/-31/53 & \textbf{F} & 7.023/-42/76 & \textbf{F} & 26.553/-71/111 & 16.726/-12/49 & \textbf{F} & 39.488/-142/135 &  & \textbf{F} & \textbf{F} & \textbf{F} \\
O-BD & $V_m{+}Q_g$ & 18 & 11 & 0.726 & 84 & 0.770 & -67.5 & 0.032/-41/58 & 0.030/15/14 & 0.102/-68/84 & 0.095/1/39 & 1.110/-538/91 & 0.280/-18/37 & 0.726/-49/80 & 2.853/-235/194 & \textbf{F} & 13.696/-178/169 & \textbf{F} & \textbf{F} & 31.980/-113/131 & \textbf{F} & 62.661/-284/231 &  & \textbf{F} & \textbf{F} & \textbf{F} \\
O-BD & $P_g{+}Q_g$ & 18 & 13 & 1.191 & 62 & 1.294 & -39.8 & 0.032/-39/62 & 0.032/10/15 & 0.095/-56/81 & 0.088/8/27 & 0.238/-36/51 & 0.223/6/20 & 2.599/-432/128 & 1.191/-40/61 & 10.647/-292/219 & 9.462/-92/103 & \textbf{F} & 38.911/-150/216 & 16.121/-8/46 & \textbf{F} & 42.320/-159/146 &  & \textbf{F} & \textbf{F} & \textbf{F} \\
O-BD & $V_a{+}V_m{+}P_g$ & 18 & 16 & 2.814 & 39 & 3.048 & -9.4 & 0.019/16/17 & 0.032/11/18 & 0.073/-20/44 & 0.074/23/17 & 0.214/-23/30 & 0.225/5/20 & 0.659/-35/70 & 0.838/2/34 & 4.790/-76/86 & 7.542/-53/84 & 14.165/-58/130 & 9.555/38/33 & 12.750/15/27 & 28.452/38/113 & 26.432/-62/80 &  & \textbf{F} & 134.318/-198/241 & \textbf{F} \\
O-BD & $V_a{+}V_m{+}Q_g$ & 18 & 11 & 0.618 & 93 & 0.633 & -109.3 & 0.018/21/12 & 0.033/7/18 & 0.066/-8/40 & 0.080/16/21 & 0.416/-139/93 & 0.618/-161/130 & 0.703/-44/75 & \textbf{F} & 12.284/-352/267 & 15.821/-221/204 & \textbf{F} & \textbf{F} & 31.363/-109/126 & \textbf{F} & 73.685/-352/277 &  & \textbf{F} & \textbf{F} & \textbf{F} \\
O-BD & $V_a{+}P_g{+}Q_g$ & 18 & 14 & 1.565 & 66 & 1.611 & -44.1 & 0.028/-24/38 & 0.037/-5/25 & 0.081/-34/57 & 0.088/8/26 & 0.267/-53/45 & 0.206/13/17 & 1.595/-226/224 & 1.535/-80/91 & 9.367/-244/198 & 6.666/-35/73 & 26.283/-193/268 & 13.486/13/60 & 24.727/-65/89 & \textbf{F} & 61.888/-279/228 &  & \textbf{F} & \textbf{F} & \textbf{F} \\
O-BD & $V_m{+}P_g{+}Q_g$ & 18 & 12 & 0.645 & 62 & 0.681 & -36.8 & 0.026/-12/45 & 0.034/5/12 & 0.079/-30/61 & 0.086/10/28 & 0.335/-92/53 & 0.215/10/16 & 0.956/-96/124 & 1.455/-71/65 & \textbf{F} & 6.969/-41/69 & \textbf{F} & 37.278/-140/167 & 19.822/-32/63 & \textbf{F} & 29.764/-82/92 &  & \textbf{F} & \textbf{F} & \textbf{F} \\
O-BD & $V_a{+}V_m{+}P_g{+}Q_g$ & 18 & 16 & 1.905 & 16 & 2.361 & 22.8 & 0.016/28/9 & 0.029/20/5 & 0.045/26/9 & 0.069/28/6 & 0.181/-4/22 & 0.197/17/9 & 0.326/33/15 & 0.706/17/22 & 5.324/-96/90 & 3.104/37/17 & 13.269/-48/109 & 4.929/68/5 & 9.756/35/11 & 9.721/79/19 & 17.340/-6/35 &  & 62.732/-107/175 & \textbf{F} & \textbf{F} \\
\midrule
\multicolumn{27}{l}{\textit{Oracle AC primal+dual all-bounds (O-PD-AB): constraint duals + ALL bound mults}} \\
O-PD-AB & $V_a$ & 18 & 17 & 1.589 & 7 & 1.990 & 31.4 & 0.023/1/6 & 0.034/5/7 & 0.043/29/5 & 0.065/31/6 & 0.131/25/5 & 0.166/30/4 & 0.620/-27/6 & 0.579/32/8 & 1.589/42/10 & 2.446/50/6 & 5.118/43/14 & 5.573/64/8 & 8.967/40/7 & 47.540/-4/191 & 10.587/35/13 &  & 25.781/15/47 & \textbf{F} & 94.997/43/75 \\
O-PD-AB & $V_m$ & 18 & 16 & 2.768 & 19 & 3.011 & -5.1 & 0.027/-16/43 & 0.034/6/6 & 0.076/-24/16 & 0.073/24/12 & 0.161/7/19 & 0.180/24/5 & 2.075/-325/63 & 0.753/12/18 & 8.541/-214/177 & 3.461/30/19 & 24.884/-177/243 & 9.798/37/16 & 11.494/23/19 & 61.624/-34/255 & 18.901/-16/44 &  & 83.035/-174/246 & \textbf{F} & \textbf{F} \\
O-PD-AB & $P_g$ & 18 & 17 & 3.215 & 30 & 3.985 & -3.2 & 0.024/-3/20 & 0.030/17/7 & 0.062/-1/35 & 0.065/32/8 & 0.169/3/18 & 0.171/28/5 & 0.624/-28/54 & 1.060/-24/28 & 13.108/-382/263 & 3.215/35/17 & 27.106/-202/277 & 16.828/-8/50 & 8.212/45/12 & 65.523/-43/245 & 15.223/7/30 &  & 89.731/-196/250 & 145.344/-223/239 & \textbf{F} \\
O-PD-AB & $Q_g$ & 18 & 14 & 1.233 & 35 & 1.402 & -4.2 & 0.023/-1/16 & 0.034/4/6 & 0.094/-54/37 & 0.065/32/7 & 0.221/-27/33 & 0.179/25/7 & 0.719/-47/39 & 1.746/-105/37 & 7.026/-158/135 & 4.823/2/39 & 25.911/-189/260 & 9.958/36/20 & 8.216/45/11 & \textbf{F} & 17.599/-8/39 &  & \textbf{F} & \textbf{F} & \textbf{F} \\
O-PD-AB & $V_a{+}V_m$ & 18 & 17 & 1.426 & 4 & 1.829 & 45.4 & 0.017/28/3 & 0.029/19/3 & 0.174/-186/3 & 0.066/31/3 & 0.127/27/3 & 0.169/29/3 & 0.267/45/3 & 0.485/43/4 & 1.426/48/6 & 2.331/53/4 & 3.269/64/10 & 6.318/59/4 & 8.476/43/5 & 6.813/85/6 & 8.482/48/5 &  & 13.954/54/10 & \textbf{F} & 50.004/70/24 \\
O-PD-AB & $V_a{+}P_g$ & 18 & 18 & 2.001 & 8 & 2.590 & 32.8 & 0.017/27/6 & 0.028/22/7 & 0.045/26/5 & 0.070/27/6 & 0.127/27/5 & 0.172/28/4 & 0.298/39/6 & 0.534/37/8 & 1.555/43/10 & 2.446/50/5 & 3.566/60/14 & 6.900/56/10 & 7.766/48/9 & 43.930/4/183 & 11.707/28/17 &  & 18.168/40/26 & 159.252/-254/259 & 57.566/65/34 \\
O-PD-AB & $V_a{+}Q_g$ & 18 & 17 & 1.644 & 7 & 2.059 & 33.4 & 0.015/33/5 & 0.030/16/5 & 0.111/-83/7 & 0.065/32/5 & 0.128/27/5 & 0.170/28/4 & 0.290/41/7 & 0.557/35/10 & 1.644/40/12 & 2.718/45/10 & 3.685/59/16 & 5.673/63/7 & 8.853/41/6 & 46.133/-1/194 & 12.934/21/21 &  & 43.541/-44/96 & \textbf{F} & 82.818/50/58 \\
O-PD-AB & $V_m{+}P_g$ & 18 & 15 & 1.148 & 21 & 1.197 & 4.7 & 0.020/14/20 & 0.028/20/8 & 0.053/13/13 & 0.066/31/7 & 0.166/5/19 & 0.174/27/5 & 0.722/-48/64 & 1.148/-35/53 & 9.134/-236/190 & 3.520/29/21 & 25.564/-185/245 & 29.422/-89/98 & 11.793/21/20 & \textbf{F} & 19.814/-21/47 &  & 101.736/-235/300 & \textbf{F} & \textbf{F} \\
O-PD-AB & $V_m{+}Q_g$ & 18 & 15 & 0.792 & 35 & 0.869 & -11.1 & 0.027/-18/41 & 0.032/11/7 & 0.189/-211/35 & 0.073/24/12 & 0.163/6/18 & 0.177/26/6 & 0.543/-11/47 & 0.792/7/26 & 9.260/-241/194 & 5.682/-15/49 & 26.798/-199/260 & 8.380/46/14 & 10.118/32/18 & 74.129/-62/276 & 24.764/-52/65 &  & \textbf{F} & \textbf{F} & \textbf{F} \\
O-PD-AB & $P_g{+}Q_g$ & 18 & 14 & 0.874 & 32 & 0.990 & -7.3 & 0.027/-17/22 & 0.029/17/8 & 0.060/2/28 & 0.065/32/8 & 0.146/16/13 & 0.184/23/8 & 0.742/-52/36 & 1.005/-18/42 & 7.352/-170/143 & 4.782/3/40 & \textbf{F} & 61.055/-293/184 & 10.198/32/12 & \textbf{F} & 25.630/-57/64 &  & 79.677/-163/234 & \textbf{F} & \textbf{F} \\
O-PD-AB & $V_a{+}V_m{+}P_g$ & 18 & 18 & 1.836 & 4 & 2.265 & 46.8 & 0.020/14/3 & 0.027/24/2 & 0.068/-13/3 & 0.064/33/3 & 0.128/26/3 & 0.168/29/3 & 0.265/46/3 & 0.479/44/4 & 1.345/51/5 & 2.326/53/4 & 4.135/54/10 & 6.623/57/4 & 7.584/49/5 & 6.824/85/6 & 8.526/48/5 &  & 13.564/55/9 & 151.529/-237/240 & 40.842/75/15 \\
O-PD-AB & $V_a{+}V_m{+}Q_g$ & 18 & 18 & 1.822 & 4 & 2.428 & 47.3 & 0.017/24/3 & 0.028/22/3 & 0.066/-9/3 & 0.062/35/3 & 0.134/23/3 & 0.162/32/3 & 0.263/46/3 & 0.488/43/3 & 1.302/52/4 & 2.341/53/4 & 3.217/64/9 & 4.952/68/4 & 8.278/45/4 & 6.551/86/4 & 8.416/48/4 &  & 11.660/62/8 & 17.204/62/5 & 43.989/73/18 \\
O-PD-AB & $V_a{+}P_g{+}Q_g$ & 18 & 17 & 1.625 & 9 & 2.027 & 35.8 & 0.020/12/5 & 0.034/3/5 & 0.068/-12/7 & 0.067/29/5 & 0.133/23/5 & 0.168/29/4 & 0.307/37/7 & 0.547/36/9 & 1.625/40/11 & 2.728/45/9 & 3.750/58/16 & 6.006/61/10 & 7.413/51/8 & 18.684/59/63 & 12.207/25/19 &  & 22.632/25/47 & \textbf{F} & 49.091/70/26 \\
O-PD-AB & $V_m{+}P_g{+}Q_g$ & 18 & 14 & 0.765 & 27 & 0.812 & -19.8 & 0.018/22/13 & 0.030/16/9 & 0.123/-102/22 & 0.068/29/9 & 0.159/9/19 & 0.174/27/6 & 0.686/-40/70 & 0.845/1/32 & 11.760/-332/242 & 7.217/-46/64 & 27.094/-202/266 & \textbf{F} & 11.389/24/17 & \textbf{F} & 24.927/-53/65 &  & \textbf{F} & 169.635/-277/283 & \textbf{F} \\
O-PD-AB & $V_a{+}V_m{+}P_g{+}Q_g$ & 18 & 18 & 1.786 & 4 & 2.389 & 50.1 & 0.017/27/2 & 0.030/15/2 & 0.099/-63/2 & 0.062/35/2 & 0.126/27/2 & 0.162/32/2 & 0.255/48/2 & 0.472/45/3 & 1.307/52/4 & 2.265/54/3 & 3.109/65/9 & 6.133/61/4 & 6.579/56/4 & 6.541/86/4 & 8.452/48/4 &  & 11.201/63/8 & 17.428/61/5 & 38.632/77/12 \\
\midrule
\multicolumn{27}{l}{\textit{Dual-only: constraint duals, no primal blocks}} \\
DO-C & --- & 18 & 13 & 1.023 & 65 & 1.059 & -39.2 & 0.027/-17/19 & 0.034/5/14 & 0.096/-58/30 & 0.081/15/18 & 0.207/-19/35 & 0.302/-27/41 & 1.023/-109/130 & 1.970/-131/104 & 8.838/-225/171 & 6.865/-39/65 & \textbf{F} & \textbf{F} & 25.834/-72/95 & 52.032/-14/211 & 37.338/-129/114 &  & \textbf{F} & \textbf{F} & \textbf{F} \\
\midrule
\multicolumn{27}{l}{\textit{Dual-only: all bound mults, no primal blocks}} \\
DO-B & --- & 18 & 14 & 1.795 & 47 & 1.859 & -36.2 & 0.051/-121/80 & 0.034/5/10 & 0.083/-37/27 & 0.086/10/24 & 0.236/-36/45 & 0.192/19/9 & 0.795/-63/97 & 2.795/-228/71 & 9.918/-265/209 & 5.562/-13/49 & \textbf{F} & 8.302/47/25 & 8.924/40/16 & \textbf{F} & 22.308/-37/62 &  & 61.616/-103/185 & \textbf{F} & \textbf{F} \\
\midrule
\multicolumn{27}{l}{\textit{Dual-only: constraint duals + all bound mults, no primal blocks}} \\
DO-F & --- & 18 & 14 & 0.732 & 24 & 0.891 & 8.6 & 0.019/17/15 & 0.031/13/7 & 0.105/-73/35 & 0.072/25/7 & 0.193/-11/32 & 0.174/27/5 & 0.523/-7/46 & 0.941/-10/28 & 8.907/-228/180 & 3.348/32/19 & 22.701/-153/228 & 9.279/40/19 & 9.781/35/11 & \textbf{F} & 15.688/4/32 &  & \textbf{F} & \textbf{F} & \textbf{F} \\
\midrule
\multicolumn{27}{l}{\textit{Practical DC-seeded (DC)}} \\
DC & $V_a$ & 18 & 17 & 2.630 & 38 & 3.239 & 3.3 & 0.022/3/15 & 0.030/15/15 & 0.099/-63/23 & 0.072/24/14 & 0.192/-10/29 & 0.233/2/21 & 0.413/15/30 & 1.408/-65/56 & 2.630/3/37 & 6.668/-35/61 & 6.529/27/51 & 15.678/-1/55 & 14.646/2/38 & \textbf{F} & 16.442/-1/39 &  & 27.365/10/64 & 42.229/6/61 & 107.184/35/110 \\
DC & $P_g$ & 18 & 17 & 2.856 & 40 & 3.434 & 0.1 & 0.023/-1/17 & 0.030/14/13 & 0.092/-51/23 & 0.072/25/12 & 0.171/2/22 & 0.227/5/19 & 0.430/12/32 & 0.856/-0/31 & 2.856/-5/42 & 5.212/-6/48 & 8.388/7/54 & 12.259/21/40 & 14.971/0/40 & \textbf{F} & 16.308/0/40 &  & 32.050/-6/78 & 43.777/3/62 & 165.297/0/176 \\
DC & $V_a{+}P_g$ & 18 & 18 & 4.008 & 34 & 4.628 & 9.1 & 0.018/21/13 & 0.032/9/14 & 0.077/-26/18 & 0.075/21/13 & 0.190/-9/28 & 0.217/8/16 & 0.418/15/30 & 1.153/-35/50 & 2.468/9/33 & 5.548/-12/48 & 7.639/15/44 & 12.021/23/38 & 11.942/20/32 & 58.667/-28/250 & 15.204/7/34 &  & 74.502/-146/195 & 37.476/17/52 & 108.666/34/107 \\
\bottomrule
\end{tabular}
\end{adjustbox}
\end{table}
\end{landscape}
\restoregeometry
\endgroup
\FloatBarrier
\clearpage


\ifextended

\appendices

\section{Supplementary Proofs}
\label{app:proofs}

\subsection{Interior-Point Initialization and Barrier Centering}
\label{app:barrier}

Interior-point methods replace the bound constraints $l \le x \le u$ with a logarithmic barrier $-\mu\sum_i \big[\log(x_i - l_i) + \log(u_i - x_i)\big]$ added to the objective and drive $\mu \to 0$. The barrier diverges as any component approaches its bound, so the barrier-augmented objective grows without bound near the feasible boundary (Fig.~\ref{fig:barrier_cartoon}). A starting point placed on or very close to a bound therefore sits where the barrier and its gradient are enormous, and the first Newton steps are dominated by the need to move back into the interior rather than toward the optimum. This is the mechanism behind the well-known observation that interior-point iteration counts are relatively insensitive to the starting point, which in turn makes these methods harder to warm-start than active-set or SQP methods that can use a high-quality approximate solution directly~\cite{forsgren2005,nocedal2006numerical}.

IPOPT addresses this by pushing any supplied primal value that lies too close to a bound a small distance into the interior, controlled by \mbox{\texttt{bound\_push}} and \mbox{\texttt{bound\_frac}} for a cold start and by \mbox{\texttt{warm\_start\_bound\_push}} and \mbox{\texttt{warm\_start\_bound\_frac}} when \mbox{\texttt{warm\_start\_init\_point=yes}}; the bound multipliers are likewise pushed off zero by \mbox{\texttt{warm\_start\_mult\_bound\_push}}~\cite{wachter2006ipopt,ipopt_options}. With these pushes set to $10^{-6}$ (Section~\ref{sec:compute}), IPOPT perturbs the supplied warm start only minimally, so stale or internally inconsistent initial values can have a large effect on the early iterates.

For any primal variable left without a user value, the modeling layer supplies a default that IPOPT then projects into the interior by the same bound-push rules. In the protocol studied here, however, no primal block is ever left blank: the omitted blocks of a partial restart are filled with the baseline case-data values (Section~\ref{sec:baseline}), so a ``partial'' restart mixes oracle values on some blocks with baseline values on others rather than leaving any block unspecified. The analogous choice for the duals is consequential: in the block-matched primal-plus-dual mode, bound multipliers are supplied only for initialized blocks, while omitted blocks retain IPOPT's default dual initialization. This mixed oracle/default dual state can be inconsistent with the supplied partial primal point, which helps explain why the all-bounds mode of Section~\ref{sec:oracle_pd} is more robust in the experiments.

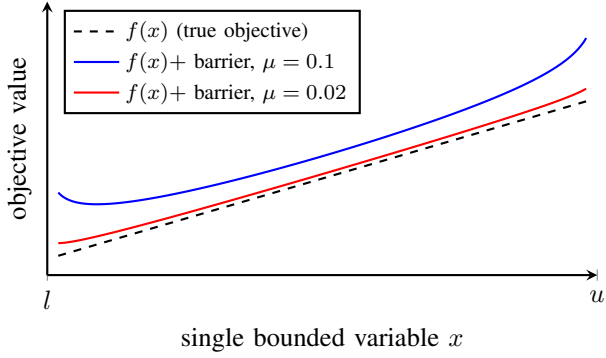
\begin{figure}[!t]
\centering
\begin{tikzpicture}
\begin{axis}[
  width=\columnwidth, height=5.2cm,
  xlabel={single bounded variable $x$},
  ylabel={objective value},
  domain=0.02:0.98, samples=200,
  xmin=0, xmax=1, ymin=-0.1, ymax=1.6,
  legend pos=north west, legend cell align=left, legend style={font=\footnotesize},
  axis lines=left, thick, xtick={0,1}, xticklabels={$l$,$u$}, ytick=\empty,
]
\addplot[black, dashed] {x}; \addlegendentry{$f(x)$ (true objective)}
\addplot[blue]  {x - 0.1*ln(x)  - 0.1*ln(1-x)};  \addlegendentry{$f(x)+$ barrier, $\mu=0.1$}
\addplot[red]   {x - 0.02*ln(x) - 0.02*ln(1-x)}; \addlegendentry{$f(x)+$ barrier, $\mu=0.02$}
\end{axis}
\end{tikzpicture}
\caption{Why interior-point initialization is delicate. For a variable with bounds $[l,u]$ and a true objective $f(x)=x$ minimized at the lower bound, the log-barrier term diverges near both bounds, so the barrier-augmented objective has an interior minimizer that approaches the active bound only as $\mu\to 0$. A start placed on the boundary must first be moved into the interior before the Newton steps can make progress toward the solution.}
\label{fig:barrier_cartoon}
\end{figure}

\subsection{Supplementary Residual Remarks}
\label{app:residuals}

The paper states the stationarity, complementarity, and barrier-centering residual bounds used to interpret partial primal-dual restarts. This appendix records two additional interpretive remarks and a feasibility-residual bound that are useful for reading the full results matrix.

\begin{remark}[Relation to standard sensitivity analysis]
The bound in Lemma~\ref{lem:residual} is an elementary consequence of the $C^2$ smoothness of the KKT system and is closely related to classical perturbation and sensitivity results for parametric nonlinear programs~\cite{forsgren2005,nocedal2006numerical,fiacco1983sensitivity}. We state it here not as a new result but to make the warm-start mechanism explicit and quantitative in the ACOPF setting, and to connect the omitted-block displacement directly to the complementarity and barrier-centering residuals (Corollaries~\ref{cor:compl}--\ref{cor:barrier}).
\end{remark}

\begin{remark}[Interpretation of the displacement bound]
\label{rem:blockwise}
Equation~\eqref{eq:sensitivitybound} is a global norm bound and does not, in general, decompose into a sum of purely block-diagonal contributions unless the relevant off-diagonal Hessian blocks vanish or are separately bounded. In standard ACOPF, some generator--voltage cross-derivative blocks are zero under the usual additive generator-injection model (cf. Remark~\ref{rem:coupling}), but substantial within-voltage and within-generator couplings remain. The practical implication is therefore qualitative rather than exact: omitted blocks with larger displacement from $x^\star$ and stronger local curvature can contribute more to the stationarity mismatch, but the contribution need not be separable block by block.
\end{remark}

\begin{corollary}[Primal feasibility residual under partial replacement]
\label{cor:feasibility}
Under the setup of Lemma~\ref{lem:residual}, suppose $g(x^\star) \approx 0$ to solver tolerance. Then for the partial restart point $\tilde{x}$,
\begin{equation}
  \|g(\tilde{x})\| \le \|g(x^\star)\| + L_g \|\tilde{x} - x^\star\|,
\end{equation}
where $L_g := \sup_{\xi \in [x^\star,\, \tilde{x}]} \|\nabla g(\xi)\|$ is the Lipschitz constant of $g$ on the segment. An analogous bound holds for the inequality constraints $c(\tilde{x})$. Thus, omitting primal blocks creates a feasibility residual proportional to the displacement, independent of any dual information.
\end{corollary}

\begin{proof}
By the mean value inequality applied componentwise to $g$,
\begin{equation}
  \|g(\tilde{x}) - g(x^\star)\| \le L_g \|\tilde{x} - x^\star\|.
\end{equation}
The triangle inequality gives the result.
\end{proof}

\begin{remark}[Feasibility and stationarity mismatch under generator-block omission]
\label{rem:coupling}
In the standard ACOPF formulation, $P_g$ and $Q_g$ enter the power-balance constraints additively, so all generator--voltage cross-derivative blocks $\partial^2\mathcal{L}/\partial V_k \partial P_g$ and $\partial^2\mathcal{L}/\partial V_k \partial Q_g$ for $V_k \in \{V_m, V_a\}$ are zero. This strict decoupling relies on the usual additive generator-injection model and need not hold under voltage-dependent capability curves, coupled generator constraints, or other formulations that mix generator and voltage variables more explicitly. The stationarity residual from Lemma~\ref{lem:residual} therefore does not propagate generator-block displacement into voltage-block residual components through Hessian coupling; indeed, because $g$ depends on the voltage blocks independently of $P_g$ and $Q_g$, the voltage-block components of $\nabla_x\mathcal{L}(\tilde{x},\nu^\star)$ remain near zero even when the generator blocks are left at baseline. However, the harm from partial restarts operates through two other channels. First, when $P_g$ and $Q_g$ are left at baseline values while $V_m$ and $V_a$ are set to their optimal values, the nonlinear AC power-balance equations $g(\tilde{x}) = 0$ are generally violated at $\tilde{x}$, creating a primal feasibility residual bounded by Corollary~\ref{cor:feasibility} even before considering dual consistency. Second, the reused constraint multipliers $(\lambda^\star,\mu_c^\star)$ correspond to $x^\star$, not to $\tilde{x}$, so the generator-side components of the stationarity residual $r(\tilde{x})$ are nonzero; in the block-matched mode the omitted generator bounds additionally carry default rather than optimal multipliers, compounding the inconsistency. Together, these feasibility, stationarity, and complementarity mismatches are consistent with the observation that $V_m{+}V_a$ with block-matched dual variables achieves a $-126.3\%$ median slowdown despite both voltage blocks being correctly initialized: the omitted generator blocks create feasibility violations and dual inconsistencies that the solver must resolve before making barrier progress.
\end{remark}

\section{Supplementary Visual Summaries}
\label{app:supplementary_summaries}

This appendix collects secondary summary displays that support the paper conclusions but are not necessary for the central argument. Fig.~\ref{fig:baseline_vs_best}, Table~\ref{tab:family_wins}, and Figs.~\ref{fig:case_size_runtime}--\ref{fig:speedup_distributions} summarize family win counts, scaling, performance profiles, ordered per-case speedups, and case-wise speedup distributions.

\subsection{Case-Wise Illustrations}
\label{sec:casewise}

Figure~\ref{fig:baseline_vs_best} shows the case-level runtime comparisons. In the left panel, the case-wise best O-PD points lie below the diagonal for most instances, which reflects the broad improvement over the baseline already seen in Table~\ref{tab:global_summary}. The best DC-seeded AC solve is more mixed: it often helps, but the gains are smaller and less consistent. The right panel shows the effect of including DCOPF presolve time. Once baseline total time is compared with the best DC end-to-end workflow, the points cluster around the diagonal rather than lying systematically below it. That visual pattern matches the non-significant matched-case E2E test in Table~\ref{tab:pairwise_tests}. Additional scaling, performance-profile, and distributional views are reported below in this appendix.

\begin{table}[!t]
\centering
\scriptsize
\setlength{\tabcolsep}{3pt}
\renewcommand{\arraystretch}{0.96}
\caption{Full ranking of all oracle AC combinations by median AC solve time. Rows are grouped by family and sorted by ascending median solve time. ``Conv.'' is the number of successful runs; case wins are tallied within each family. Median and mean solve times are in seconds.}
\label{tab:combo_rank_oracle}
\begin{tabular}{lrrrrr}
\toprule
Combination & Conv. & Med.\ solve & Mean solve & Med.\ iter & Wins \\
\midrule
\multicolumn{6}{l}{\textit{Oracle AC primal-only}} \\
$Q_g{+}V_m$ & 16 & 1.837 & 14.853 & 37 & 0 \\
$V_m$ & 17 & 2.820 & 25.840 & 40 & 0 \\
$P_g{+}Q_g$ & 17 & 2.948 & 20.489 & 36 & 1 \\
$P_g{+}V_m{+}V_a$ & 18 & 3.572 & 13.750 & 30 & 3 \\
$P_g{+}Q_g{+}V_m$ & 18 & 3.735 & 18.547 & 34 & 0 \\
$Q_g{+}V_m{+}V_a$ & 18 & 3.806 & 14.044 & 36 & 1 \\
$P_g{+}Q_g{+}V_m{+}V_a$ & 18 & 3.872 & 13.524 & 30 & 6 \\
$P_g{+}Q_g{+}V_a$ & 18 & 4.151 & 18.448 & 32 & 2 \\
$P_g{+}V_a$ & 18 & 4.282 & 16.598 & 32 & 2 \\
$Q_g$ & 18 & 4.306 & 57.391 & 40 & 0 \\
$P_g$ & 18 & 4.369 & 19.713 & 37 & 1 \\
$P_g{+}V_m$ & 18 & 4.441 & 22.868 & 34 & 1 \\
$V_a$ & 18 & 4.545 & 22.280 & 38 & 0 \\
$Q_g{+}V_a$ & 18 & 4.561 & 22.697 & 38 & 1 \\
$V_m{+}V_a$ & 18 & 4.638 & 13.819 & 36 & 0 \\
\midrule
\multicolumn{6}{l}{\textit{Oracle AC primal+dual}} \\
$V_m$ & 10 & 0.270 & 8.090 & 54 & 0 \\
$Q_g{+}V_m$ & 10 & 0.271 & 7.466 & 66 & 0 \\
$Q_g{+}V_m{+}V_a$ & 10 & 0.281 & 8.319 & 58 & 0 \\
$Q_g{+}V_a$ & 11 & 0.292 & 9.037 & 87 & 0 \\
$Q_g$ & 10 & 0.301 & 9.678 & 60 & 0 \\
$V_m{+}V_a$ & 11 & 0.567 & 8.697 & 108 & 0 \\
$P_g{+}Q_g$ & 13 & 0.614 & 15.527 & 58 & 0 \\
$P_g{+}Q_g{+}V_m$ & 14 & 0.753 & 18.396 & 27 & 1 \\
$P_g{+}V_m$ & 14 & 0.822 & 11.835 & 46 & 0 \\
$P_g$ & 13 & 1.260 & 19.946 & 58 & 0 \\
$P_g{+}V_m{+}V_a$ & 15 & 1.273 & 6.986 & 31 & 0 \\
$V_a$ & 12 & 1.298 & 12.984 & 116 & 0 \\
$P_g{+}V_a$ & 15 & 1.369 & 11.100 & 57 & 0 \\
$P_g{+}Q_g{+}V_m{+}V_a$ & 18 & 1.804 & 5.868 & 4 & 16 \\
$P_g{+}Q_g{+}V_a$ & 14 & 1.983 & 6.658 & 50 & 1 \\
\midrule
\multicolumn{6}{l}{\textit{Oracle AC constraint-dual-only}} \\
$Q_g{+}V_m{+}V_a$ & 12 & 0.495 & 15.128 & 73 & 0 \\
$Q_g$ & 13 & 0.612 & 12.246 & 63 & 0 \\
$Q_g{+}V_a$ & 12 & 0.909 & 7.865 & 72 & 0 \\
$V_m{+}V_a$ & 13 & 0.925 & 14.999 & 90 & 0 \\
$P_g{+}V_m$ & 14 & 0.989 & 6.323 & 24 & 0 \\
$P_g{+}Q_g{+}V_a$ & 15 & 1.236 & 6.793 & 33 & 0 \\
$P_g{+}Q_g$ & 14 & 1.265 & 8.128 & 50 & 0 \\
$V_a$ & 14 & 1.339 & 15.428 & 89 & 0 \\
$V_m$ & 14 & 1.441 & 16.808 & 81 & 0 \\
$Q_g{+}V_m$ & 14 & 1.543 & 17.679 & 85 & 0 \\
$P_g$ & 15 & 1.635 & 11.753 & 42 & 0 \\
$P_g{+}V_m{+}V_a$ & 16 & 1.756 & 9.775 & 20 & 0 \\
$P_g{+}Q_g{+}V_m{+}V_a$ & 18 & 1.944 & 6.164 & 6 & 18 \\
$P_g{+}V_a$ & 17 & 3.013 & 18.769 & 33 & 0 \\
$P_g{+}Q_g{+}V_m$ & 16 & 3.724 & 13.762 & 42 & 0 \\
\midrule
\multicolumn{6}{l}{\textit{Oracle AC bounds-dual-only}} \\
$V_m$ & 11 & 0.301 & 10.458 & 99 & 0 \\
$Q_g$ & 10 & 0.323 & 6.273 & 83 & 0 \\
$V_a$ & 11 & 0.373 & 7.088 & 104 & 0 \\
$Q_g{+}V_a$ & 11 & 0.517 & 7.727 & 72 & 0 \\
$Q_g{+}V_m{+}V_a$ & 11 & 0.618 & 12.281 & 93 & 0 \\
$P_g{+}Q_g{+}V_m$ & 12 & 0.645 & 8.085 & 62 & 0 \\
$P_g{+}V_m$ & 12 & 0.715 & 7.745 & 74 & 0 \\
$Q_g{+}V_m$ & 11 & 0.726 & 10.324 & 84 & 0 \\
$V_m{+}V_a$ & 13 & 0.752 & 14.386 & 97 & 0 \\
$P_g{+}Q_g$ & 13 & 1.191 & 9.381 & 62 & 0 \\
$P_g$ & 13 & 1.521 & 7.473 & 74 & 0 \\
$P_g{+}Q_g{+}V_a$ & 14 & 1.565 & 10.447 & 66 & 0 \\
$P_g{+}V_a$ & 14 & 1.611 & 10.016 & 75 & 0 \\
$P_g{+}Q_g{+}V_m{+}V_a$ & 16 & 1.905 & 7.984 & 16 & 15 \\
$P_g{+}V_m{+}V_a$ & 16 & 2.814 & 15.009 & 39 & 2 \\
\bottomrule
\end{tabular}
\end{table}

\begin{table}[!t]
\centering
\scriptsize
\setlength{\tabcolsep}{3pt}
\renewcommand{\arraystretch}{0.96}
\addtocounter{table}{-1}
\caption{Full ranking of all oracle AC combinations by median AC solve time (continued). Median and mean solve times are in seconds.}
\begin{tabular}{lrrrrr}
\toprule
Combination & Conv. & Med.\ solve & Mean solve & Med.\ iter & Wins \\

\midrule
\multicolumn{6}{l}{\textit{Oracle AC primal+dual (all bounds)}} \\
$P_g{+}Q_g{+}V_m$ & 14 & 0.765 & 18.152 & 27 & 0 \\
$Q_g{+}V_m$ & 15 & 0.792 & 10.742 & 35 & 0 \\
$P_g{+}Q_g$ & 14 & 0.874 & 13.639 & 32 & 0 \\
$P_g{+}V_m$ & 15 & 1.148 & 13.557 & 21 & 0 \\
$Q_g$ & 14 & 1.233 & 5.472 & 35 & 0 \\
$V_m{+}V_a$ & 17 & 1.426 & 6.024 & 4 & 0 \\
$V_a$ & 17 & 1.589 & 12.015 & 7 & 1 \\
$P_g{+}Q_g{+}V_a$ & 17 & 1.625 & 7.381 & 9 & 0 \\
$Q_g{+}V_a$ & 17 & 1.644 & 12.316 & 7 & 1 \\
$P_g{+}Q_g{+}V_m{+}V_a$ & 18 & 1.786 & 5.715 & 4 & 9 \\
$Q_g{+}V_m{+}V_a$ & 18 & 1.822 & 6.063 & 4 & 6 \\
$P_g{+}V_m{+}V_a$ & 18 & 1.836 & 13.584 & 4 & 1 \\
$P_g{+}V_a$ & 18 & 2.001 & 17.453 & 8 & 0 \\
$V_m$ & 16 & 2.768 & 14.070 & 19 & 0 \\
$P_g$ & 17 & 3.215 & 22.735 & 30 & 0 \\
\bottomrule
\end{tabular}
\end{table}

\begin{figure*}[!t]
\centering
\IfFileExists{figs/fig_09_baseline_vs_best_scatter.pdf}{%
  \includegraphics[width=0.7\textwidth]{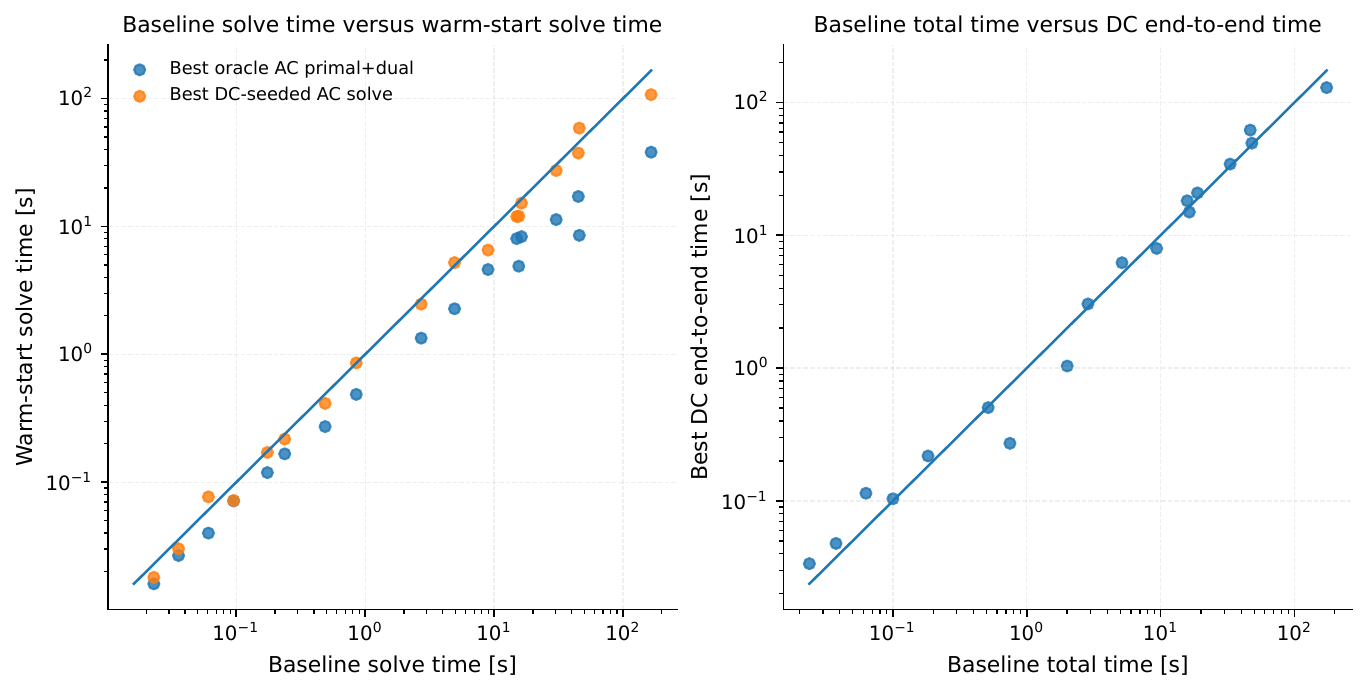}%
}{%
  \IfFileExists{figs/fig_09_baseline_vs_best_scatter.png}{%
    \includegraphics[width=0.7\textwidth]{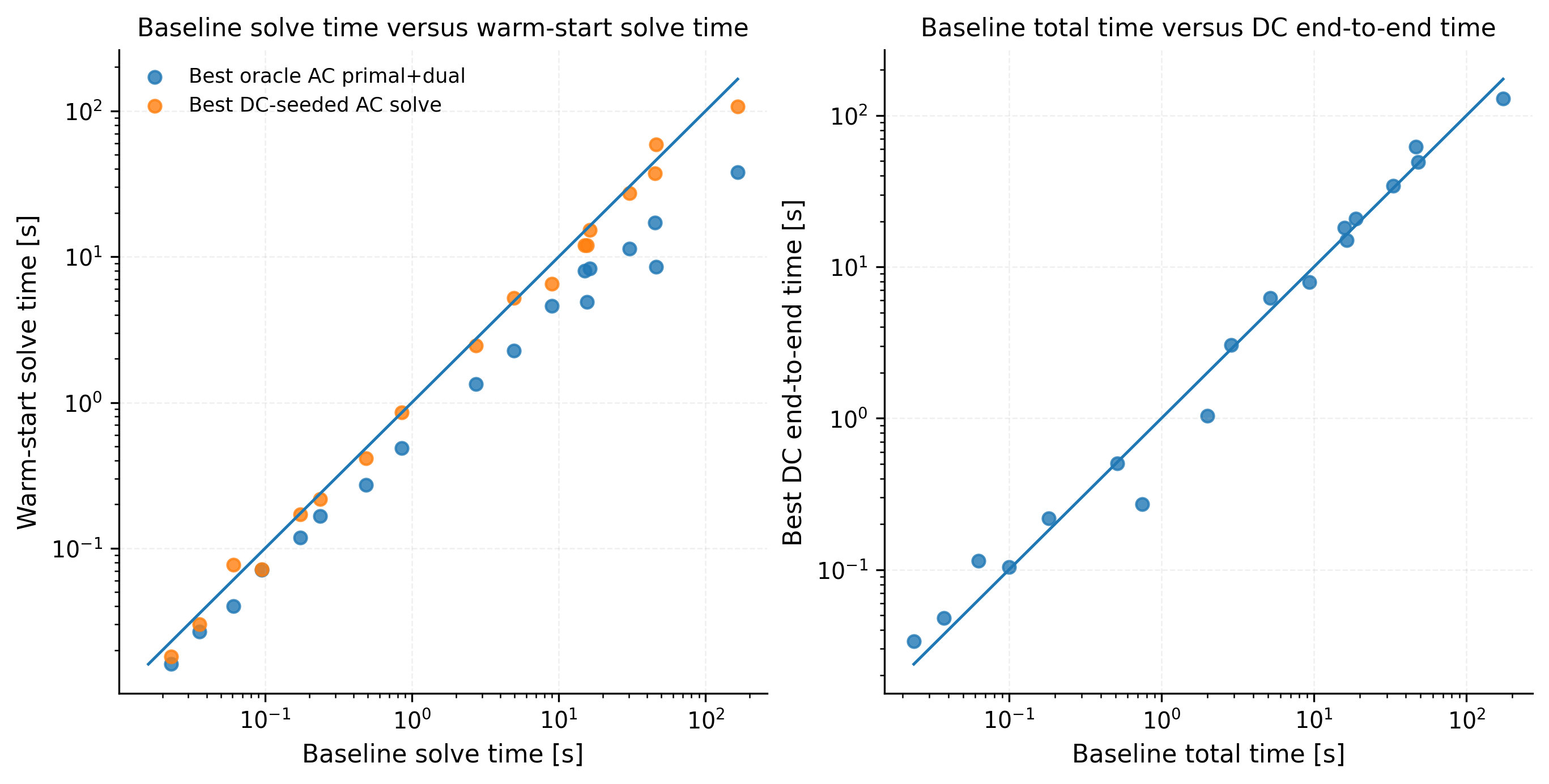}%
  }
}
\caption{Left: Baseline AC solve time versus case-wise best O-PD and case-wise best DC-seeded AC solve. Right: Baseline total time versus case-wise best DC end-to-end time.}
\label{fig:baseline_vs_best}
\end{figure*}

\begin{table}[!t]
\centering
\caption{Case-wise family win counts by best AC solve time across all families.}
\label{tab:family_wins}
\small
\begin{tabular}{lr}
\toprule
Family & Case wins \\
\midrule
Oracle AC primal-only (O-PO) & 0 \\
Block-matched primal+dual (O-PD) & 8 \\
Constraint-dual-only (O-CD) & 1 \\
Bounds-dual-only (O-BD) & 0 \\
All-bounds primal+dual (O-PD-AB) & 9 \\
DC-seeded & 0 \\
\bottomrule
\end{tabular}
\end{table}

\begin{figure}[!t]
\centering
\IfFileExists{figs/fig_01_case_size_vs_runtime.pdf}{%
  \includegraphics[width=0.98\columnwidth]{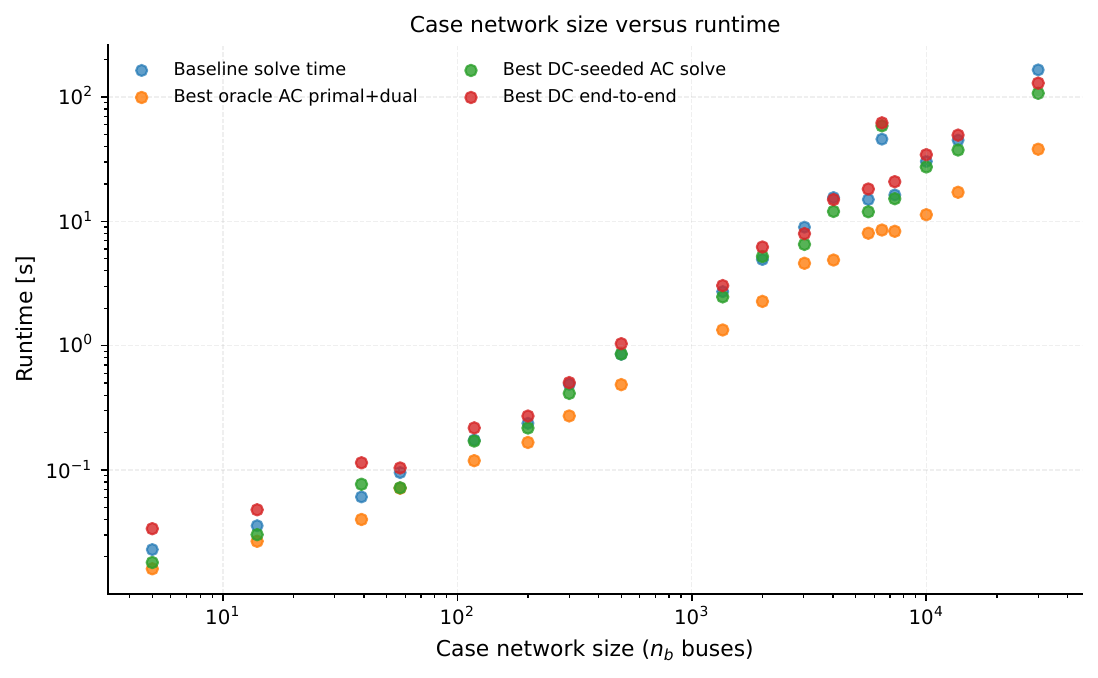}%
}{%
  \IfFileExists{figs/fig_01_case_size_vs_runtime.png}{%
    \includegraphics[width=0.98\columnwidth]{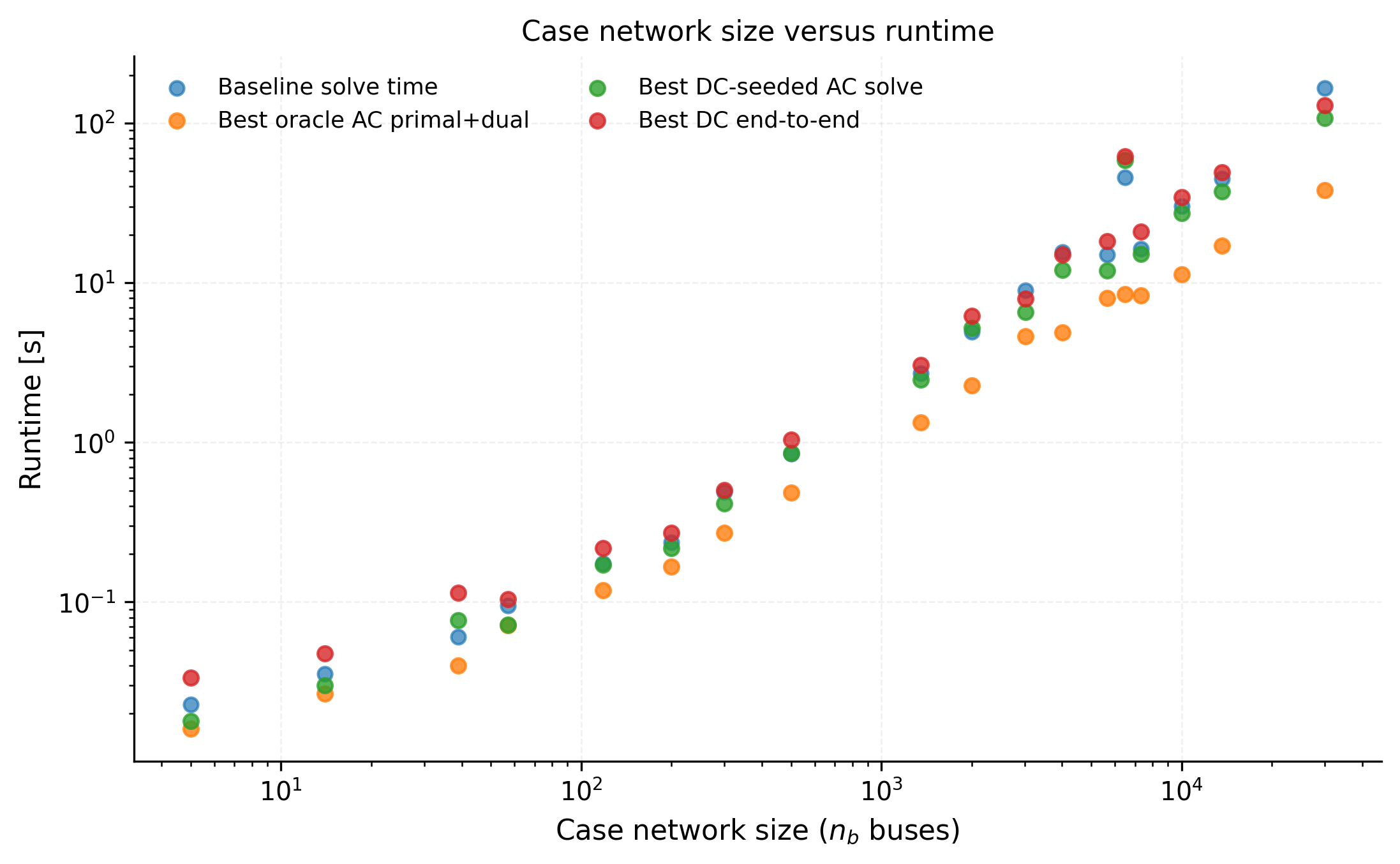}%
  }
}
\caption{Case network size ($n_b$ buses, log scale) versus runtime (log scale) for the baseline, best oracle AC primal-plus-dual restart, best DC-seeded AC solve, and best DC end-to-end workflow.}
\label{fig:case_size_runtime}
\end{figure}

\begin{figure}[!t]
\centering
\IfFileExists{figs/fig_02_performance_profile.pdf}{%
  \includegraphics[width=0.98\columnwidth]{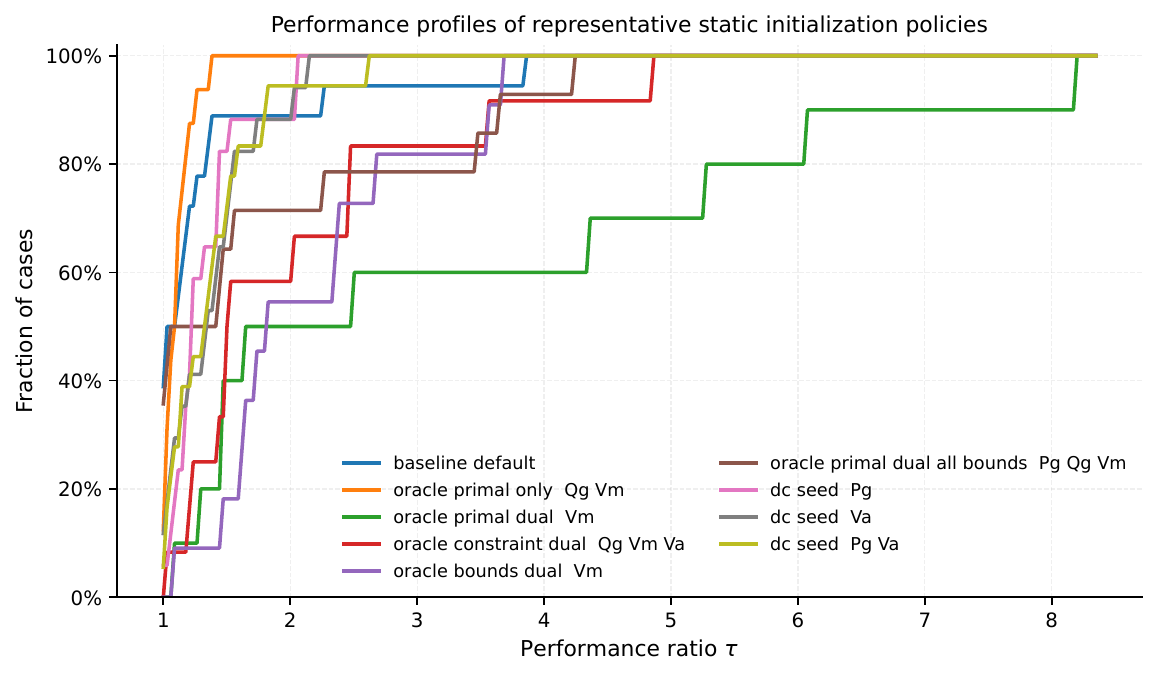}%
}{%
  \IfFileExists{figs/fig_02_performance_profile.png}{%
    \includegraphics[width=0.98\columnwidth]{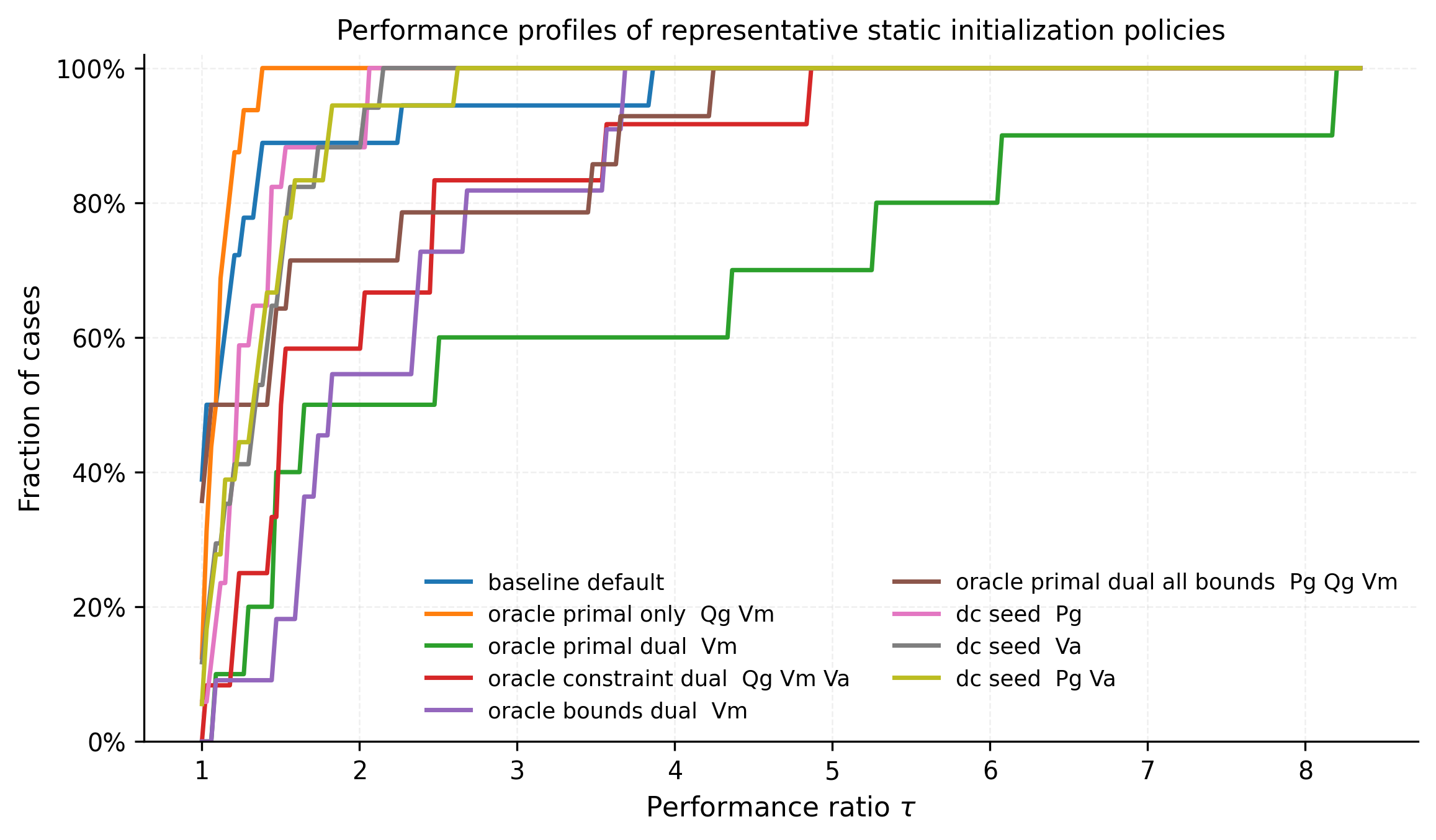}%
  }
}
\caption{Performance profiles~\cite{dolan2002benchmarking} of representative static initialization policies across all families. Failed cases are assigned ratio $\infty$.}
\label{fig:performance_profile}
\end{figure}

\begin{figure}[!t]
\centering
\IfFileExists{figs/fig_03_casewise_speedups.pdf}{%
  \includegraphics[width=0.98\columnwidth]{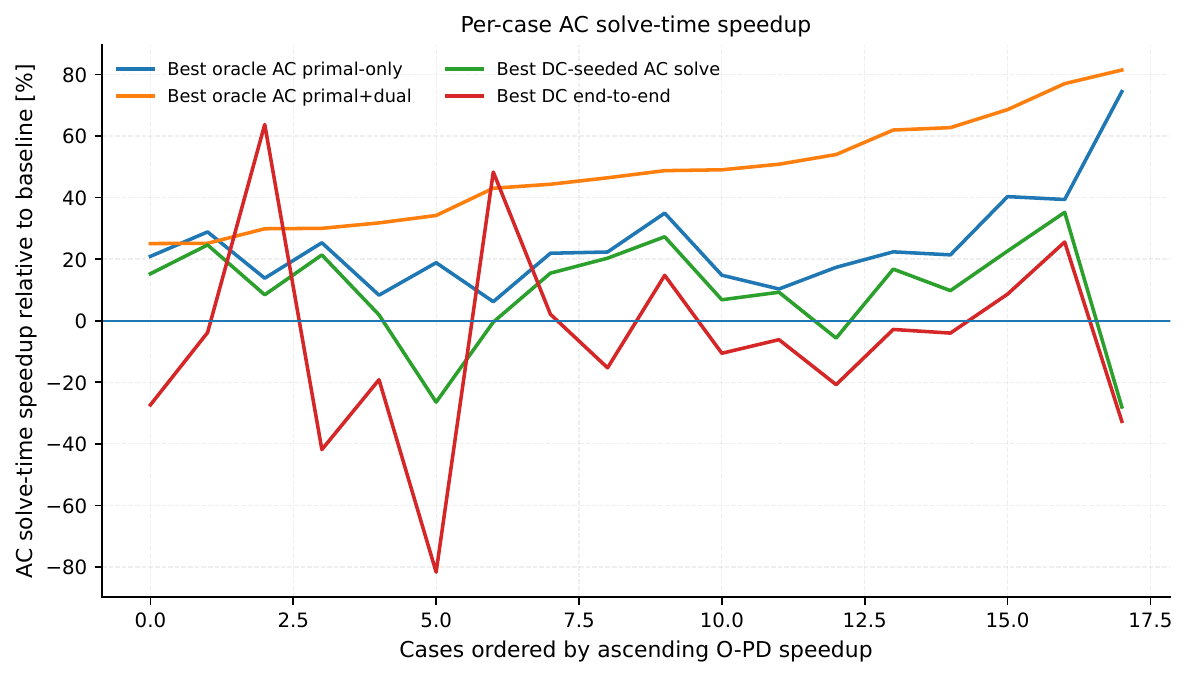}%
}{%
  \IfFileExists{figs/fig_03_casewise_speedups.png}{%
    \includegraphics[width=0.98\columnwidth]{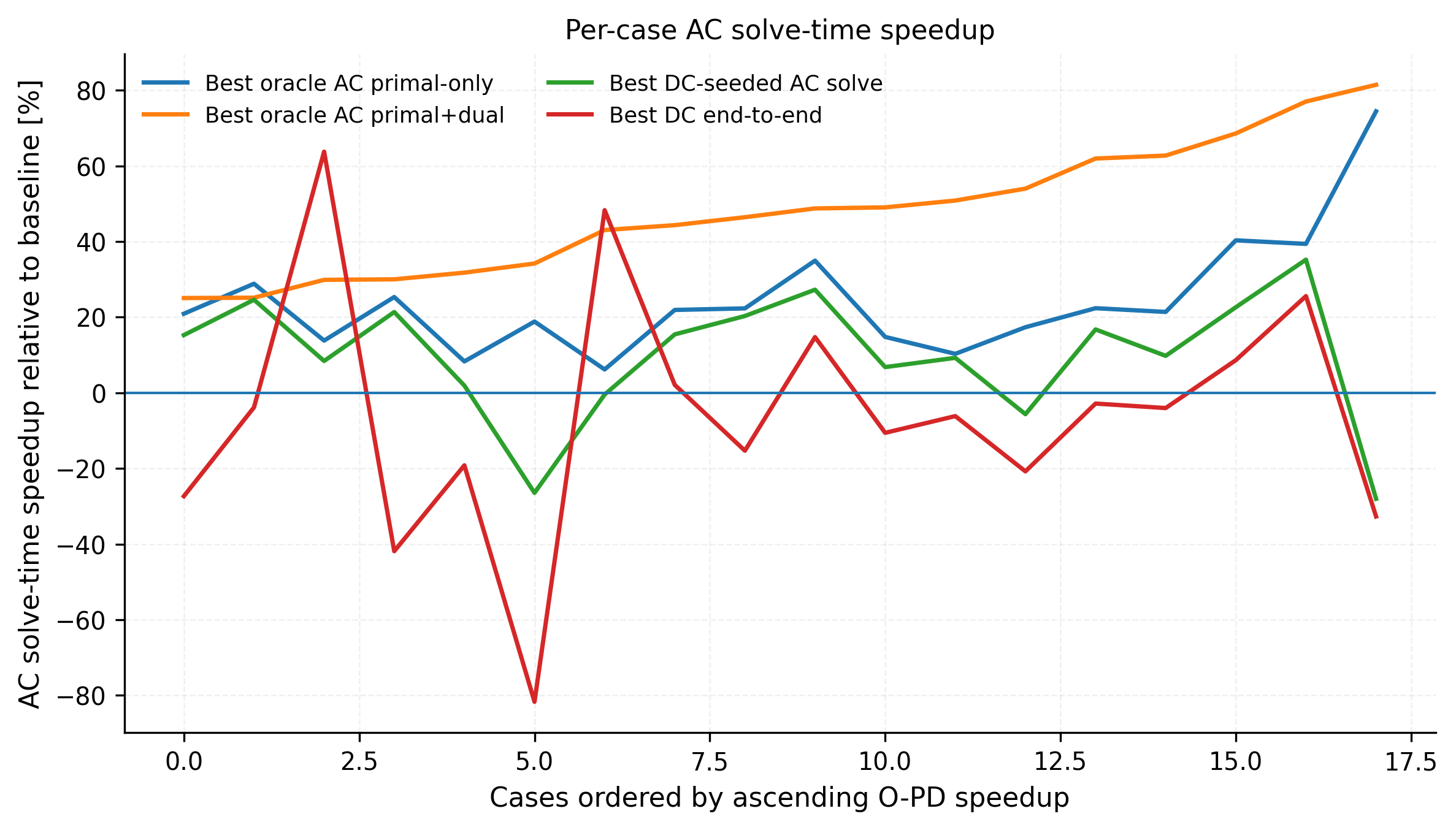}%
  }
}
\caption{Per-case AC solve-time speedup ordered by ascending O-PD speedup.}
\label{fig:casewise_speedups}
\end{figure}

\begin{figure}[!t]
\centering
\IfFileExists{figs/fig_08_speedup_distributions.pdf}{%
  \includegraphics[width=0.95\columnwidth]{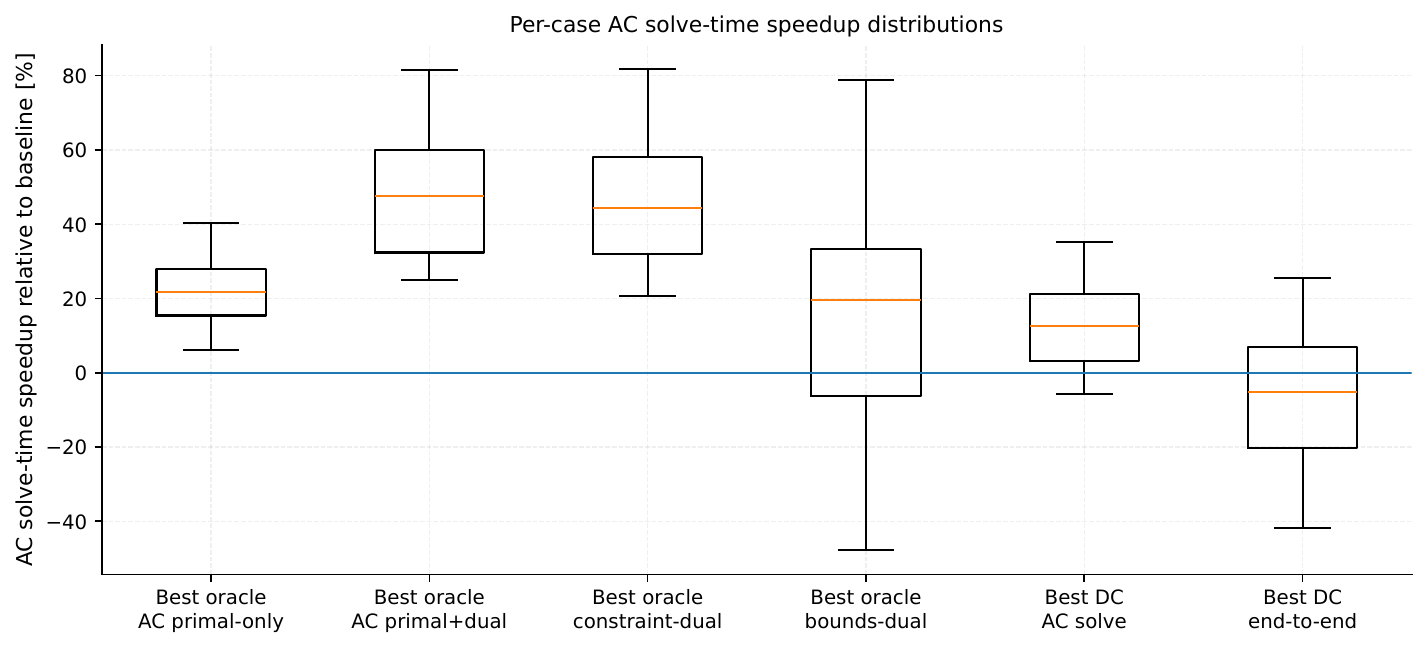}%
}{%
  \IfFileExists{figs/fig_08_speedup_distributions.png}{%
    \includegraphics[width=0.95\columnwidth]{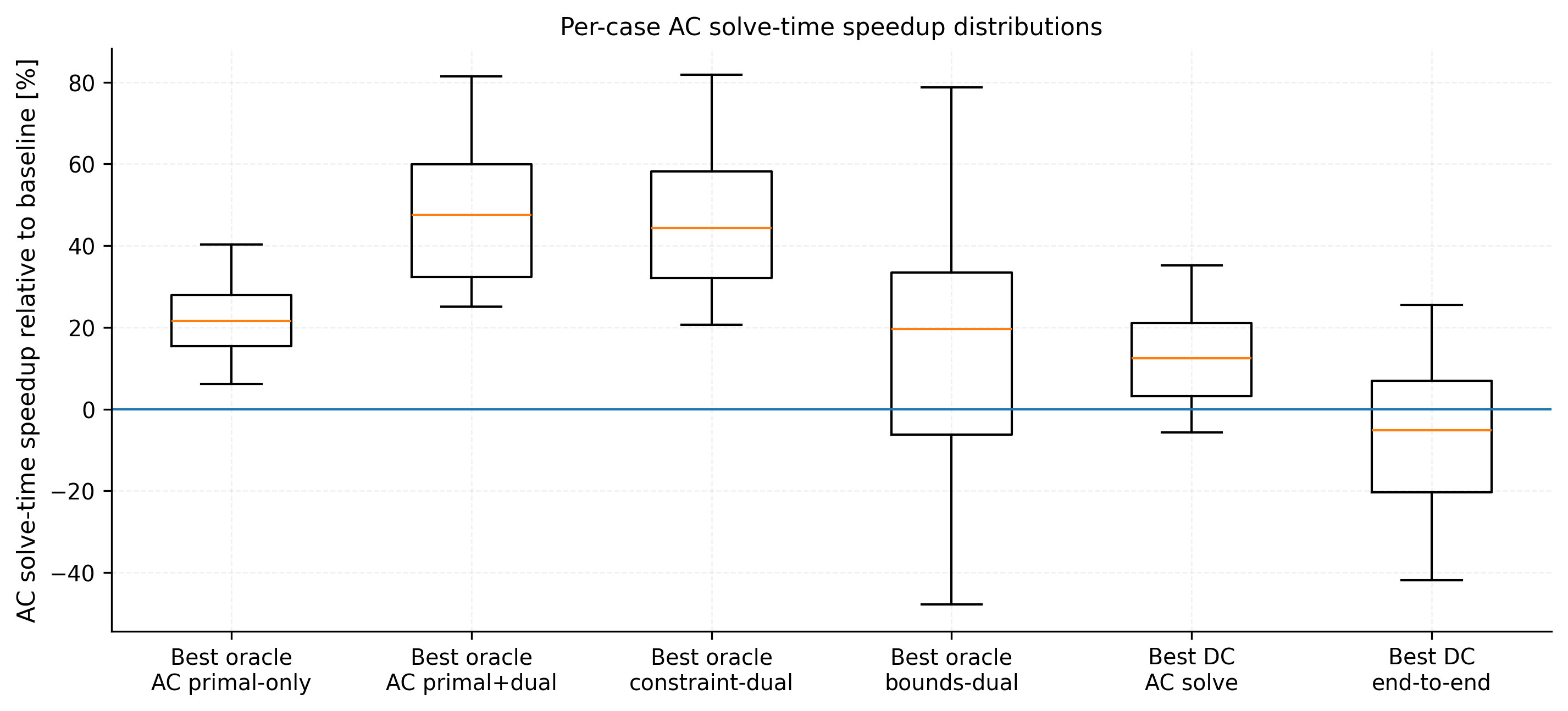}%
  }
}
\caption{Box plots of per-case AC solve-time speedup for the case-wise best envelope of each family.}
\label{fig:speedup_distributions}
\end{figure}

\section{Marginal Block Diagnostics}
\label{app:block_diagnostics}

Table~\ref{tab:block_effects} and Fig.~\ref{fig:block_benefits} report blockwise marginals. They are reported for completeness, but the paper emphasizes combination-level evidence because these marginals average over interaction effects, case-set changes due to failures, and large-case leverage.

\begin{table}[!t]
\centering
\caption{Marginal AC solve-time statistics for each initialization block. “Avg.\ reduction [\%]” is computed as $(t_{\text{excl}} - t_{\text{incl}})/t_{\text{excl}} \times 100$.}
\label{tab:block_effects}
\small
\setlength{\tabcolsep}{5pt}
\begin{adjustbox}{width=\columnwidth}
\begin{tabular}{p{5cm}lrrr}
\toprule
Family & Block & Avg.\ incl.\ [s] & Avg.\ excl.\ [s] & Avg.\ reduction [\%] \\
\midrule
Oracle AC primal-only & $V_a$ & 16.895 & 25.890 & 34.7 \\
Oracle AC primal-only & $V_m$ & 17.127 & 25.413 & 32.6 \\
Oracle AC primal-only & $P_g$ & 17.975 & 24.562 & 26.8 \\
Oracle AC primal-only & $Q_g$ & 22.622 & 19.214 & -17.7 \\
\midrule
Oracle AC primal+dual & $V_a$ & 8.530 & 13.533 & 37.0 \\
Oracle AC primal+dual & $V_m$ & 9.491 & 12.192 & 22.2 \\
Oracle AC primal+dual & $P_g$ & 11.677 & 9.276 & -25.9 \\
Oracle AC primal+dual & $Q_g$ & 10.123 & 11.429 & 11.4 \\
\midrule
Oracle AC constraint-dual-only & $V_a$ & 11.754 & 12.408 & 5.3 \\
Oracle AC constraint-dual-only & $V_m$ & 12.268 & 11.806 & -3.9 \\
Oracle AC constraint-dual-only & $P_g$ & 10.297 & 14.445 & 28.7 \\
Oracle AC constraint-dual-only & $Q_g$ & 10.785 & 13.462 & 19.9 \\
\midrule
Oracle AC bounds-dual-only & $V_a$ & 10.749 & 8.541 & -25.9 \\
Oracle AC bounds-dual-only & $V_m$ & 10.868 & 8.503 & -27.8 \\
Oracle AC bounds-dual-only & $P_g$ & 9.667 & 9.954 & 2.9 \\
Oracle AC bounds-dual-only & $Q_g$ & 9.075 & 10.561 & 14.1 \\
\midrule
Oracle AC primal+dual (all bounds) & $V_a$ & 10.087 & 14.265 & 29.3 \\
Oracle AC primal+dual (all bounds) & $V_m$ & 10.707 & 13.222 & 19.0 \\
Oracle AC primal+dual (all bounds) & $P_g$ & 13.908 & 9.544 & -45.7 \\
Oracle AC primal+dual (all bounds) & $Q_g$ & 9.682 & 14.240 & 32.0 \\
\midrule
DC-seeded & $V_a$ & 16.519 & 17.825 & 7.3 \\
DC-seeded & $V_m$ & --- & 16.946 & --- \\
DC-seeded & $P_g$ & 18.267 & 14.226 & -28.4 \\
DC-seeded & $Q_g$ & --- & 16.946 & --- \\
\bottomrule
\end{tabular}
\end{adjustbox}
\end{table}
\begin{figure}[!t]
\centering
\IfFileExists{figs/fig_06_block_effects.pdf}{%
  \includegraphics[width=0.98\columnwidth]{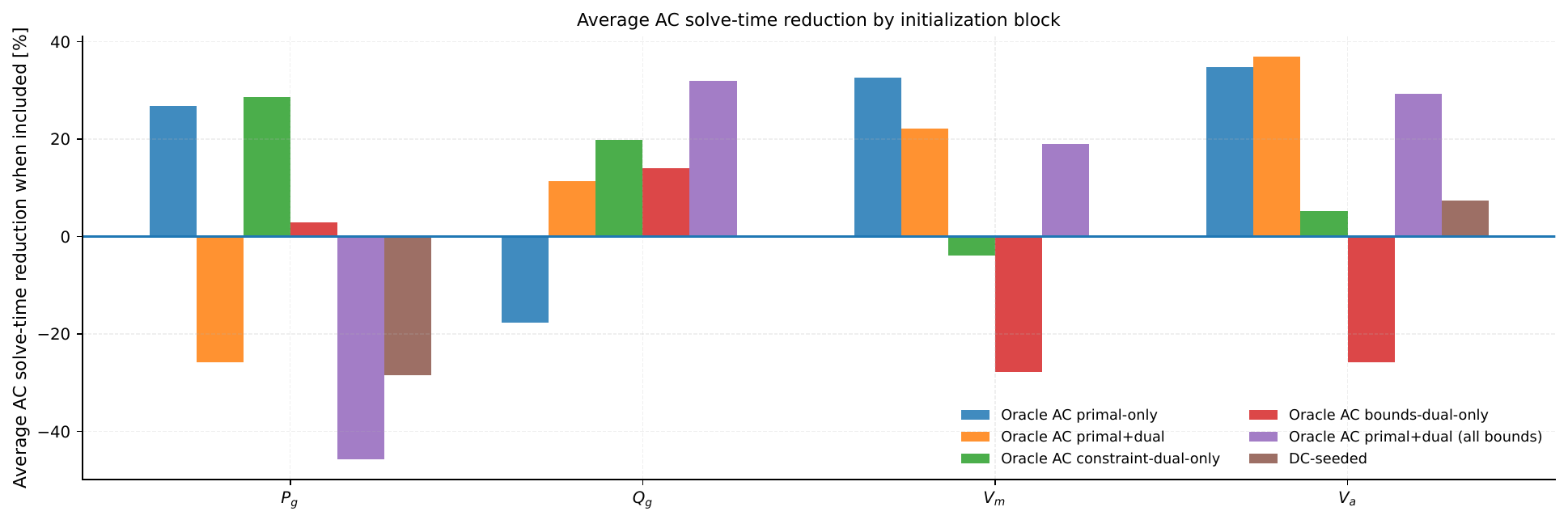}%
}{%
  \IfFileExists{figs/fig_06_block_effects.png}{%
    \includegraphics[width=0.98\columnwidth]{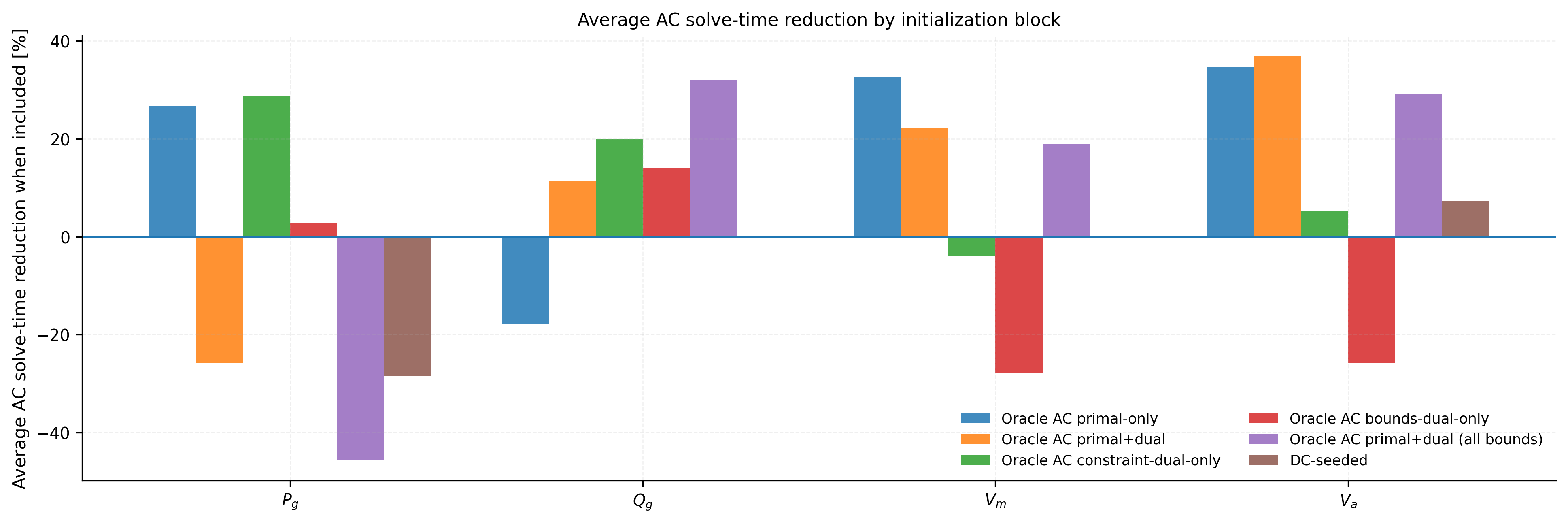}%
  }
}
\caption{Average AC solve-time reduction (\%) associated with including each initialization block, shown for all families.}
\label{fig:block_benefits}
\end{figure}

\section{Additional Dual-Only, DC, and Scaling Results}
\label{app:dual_dc_scaling}

This appendix collects supporting material for the dual-only controls, the DC-seeded family, and the speedup-versus-size view. Figure~\ref{fig:dual_only_bar}, Table~\ref{tab:combo_rank_dc}, and Figs.~\ref{fig:dc_native_boxplot} and \ref{fig:speedup_vs_case_size} summarize the dual-only ablations, the static DC ranking, the E2E distribution for DC seeds, and the cross-case size scaling.

\begin{figure}[!t]
\centering
\IfFileExists{figs/fig_14_dual_only_bar.pdf}{%
  \includegraphics[width=0.95\columnwidth]{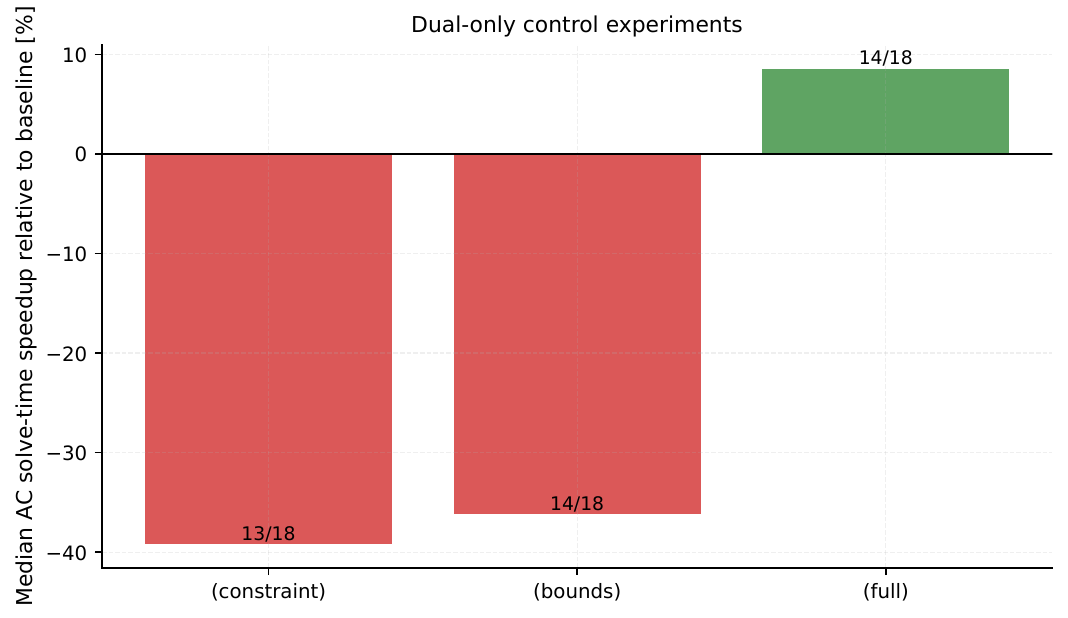}%
}{%
  \IfFileExists{figs/fig_14_dual_only_bar.png}{%
    \includegraphics[width=0.95\columnwidth]{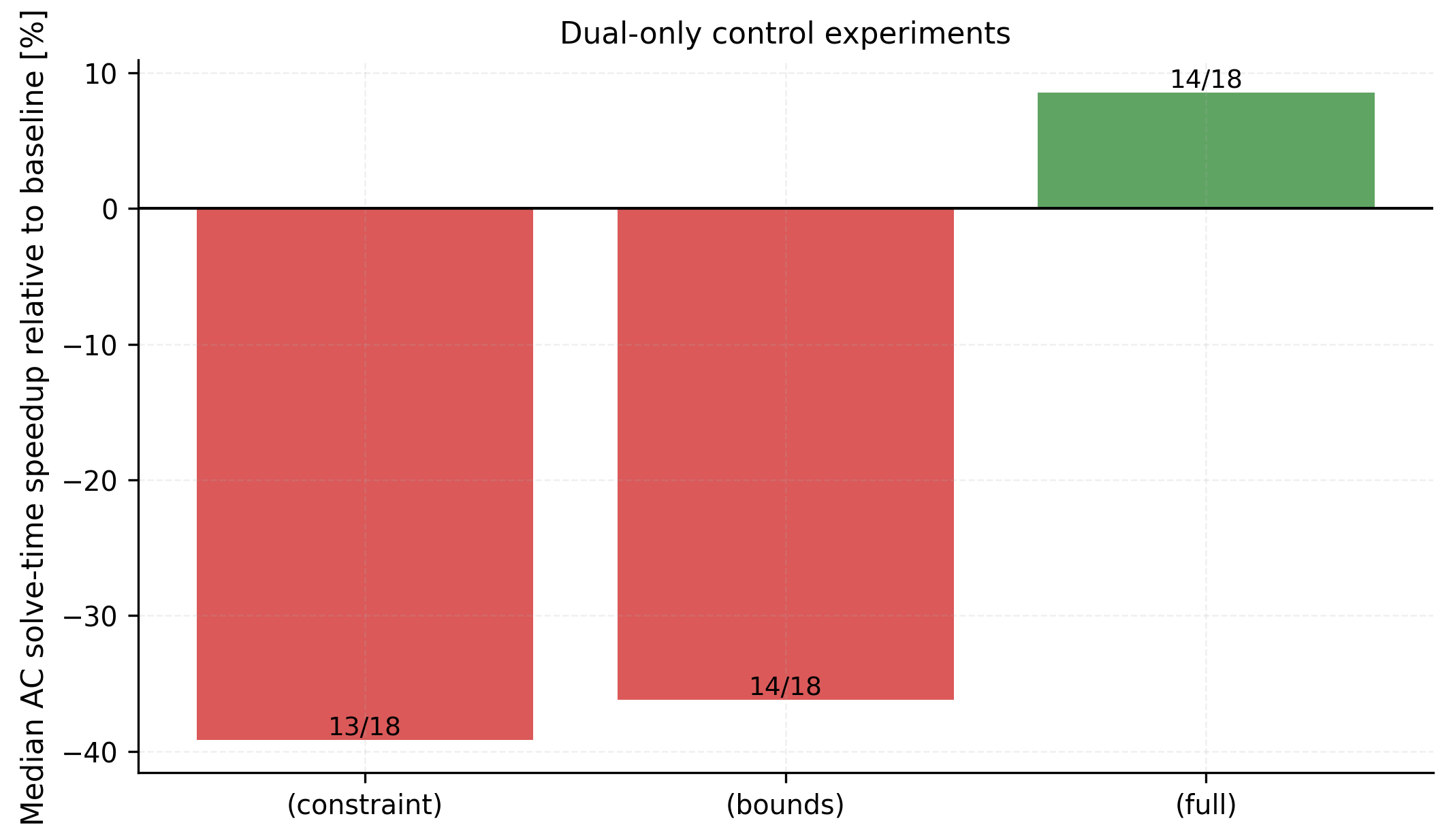}%
  }
}
\caption{Dual-only control experiments (no primal blocks): median AC solve-time speedup relative to baseline. Convergence counts are shown at the base of each bar.}
\label{fig:dual_only_bar}
\end{figure}

\begin{table}[!t]
\centering
\caption{Static ranking of DC-seeded combinations by median end-to-end time. The DCOPF presolve cost is common to all three combinations.}
\label{tab:combo_rank_dc}
\small
\setlength{\tabcolsep}{6pt}
\begin{adjustbox}{width=\columnwidth}
\begin{tabular}{lrrrrr}
\toprule
Combination & Converged & Median solve [s] & Median E2E [s] & Median iter & Case wins \\
\midrule
$V_a$ & 17 & 2.630 & 3.239 & 38 & 3 \\
$P_g$ & 17 & 2.856 & 3.434 & 40 & 6 \\
$P_g{+}V_a$ & 18 & 4.008 & 4.628 & 34 & 9 \\
\bottomrule
\end{tabular}
\end{adjustbox}
\end{table}

\begin{figure}[!t]
\centering
\IfFileExists{figs/fig_12_dc_native_boxplot.pdf}{%
  \includegraphics[width=0.95\columnwidth]{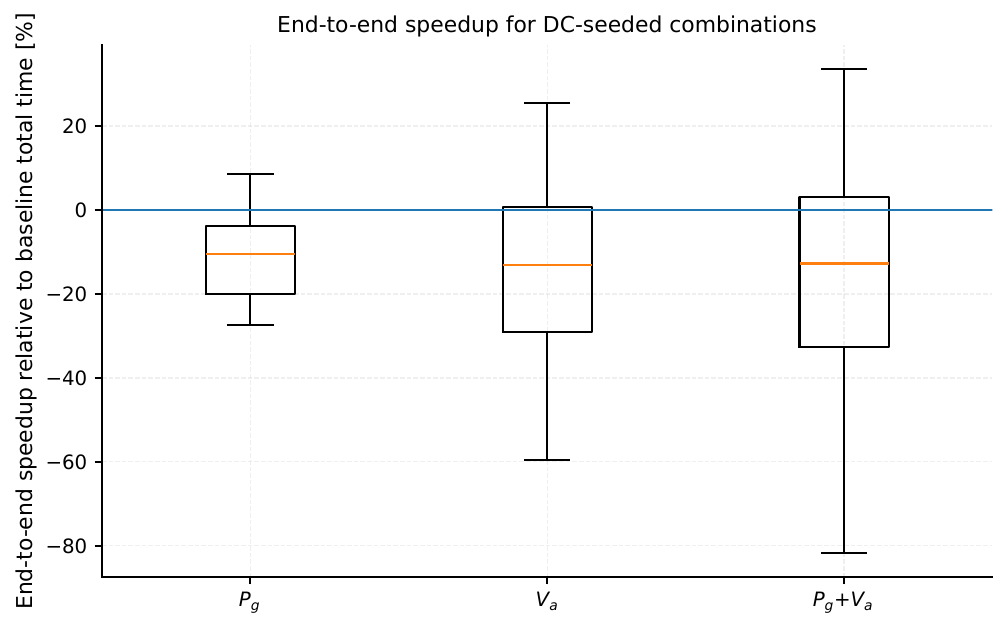}%
}{%
  \IfFileExists{figs/fig_12_dc_native_boxplot.png}{%
    \includegraphics[width=0.95\columnwidth]{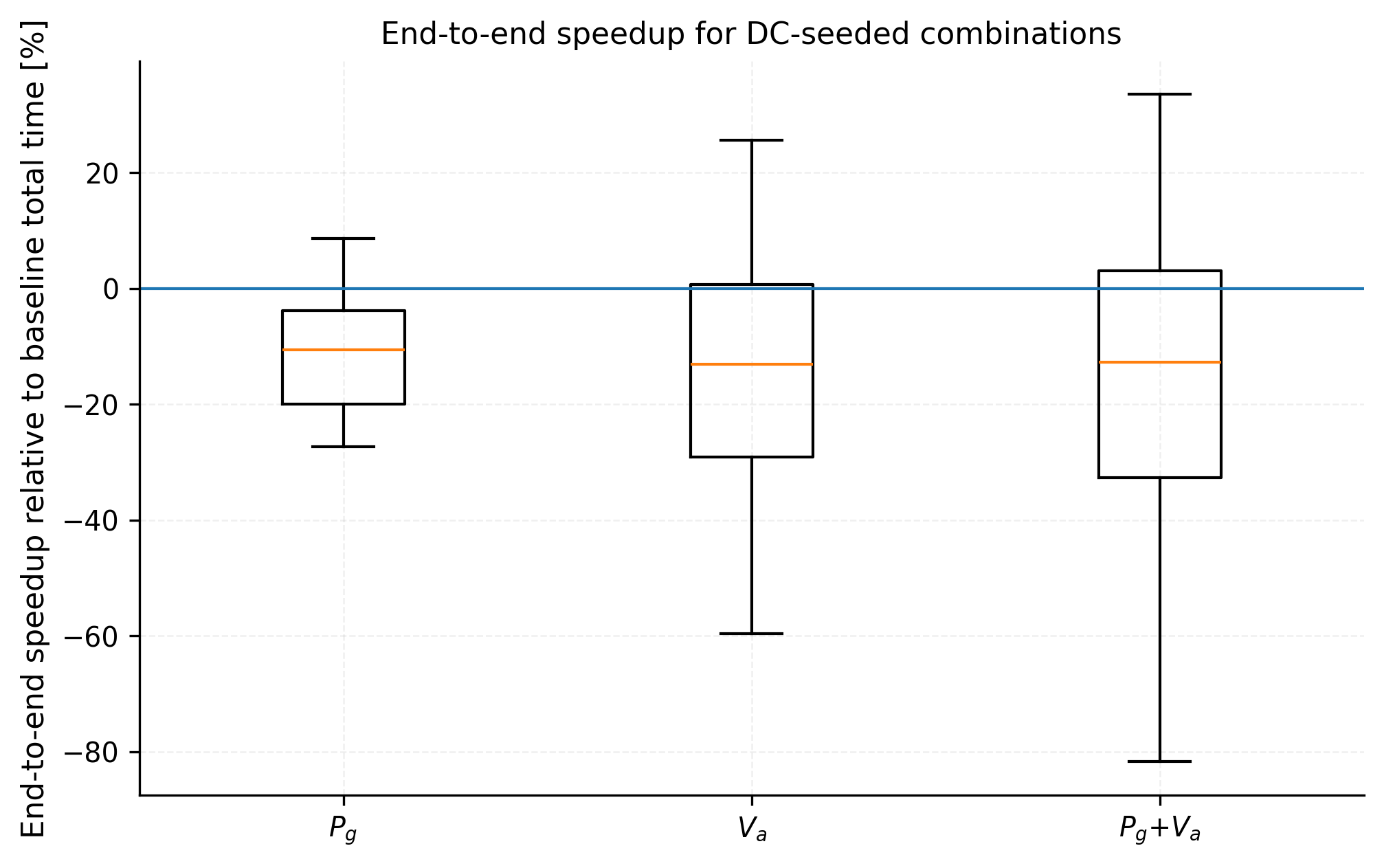}%
  }
}
\caption{Box plots of end-to-end speedup (\%) relative to baseline total time for the three DC-seeded combinations.}
\label{fig:dc_native_boxplot}
\end{figure}

\begin{figure}[!t]
\centering
\IfFileExists{figs/fig_07_speedup_vs_case_size.pdf}{%
  \includegraphics[width=0.95\columnwidth]{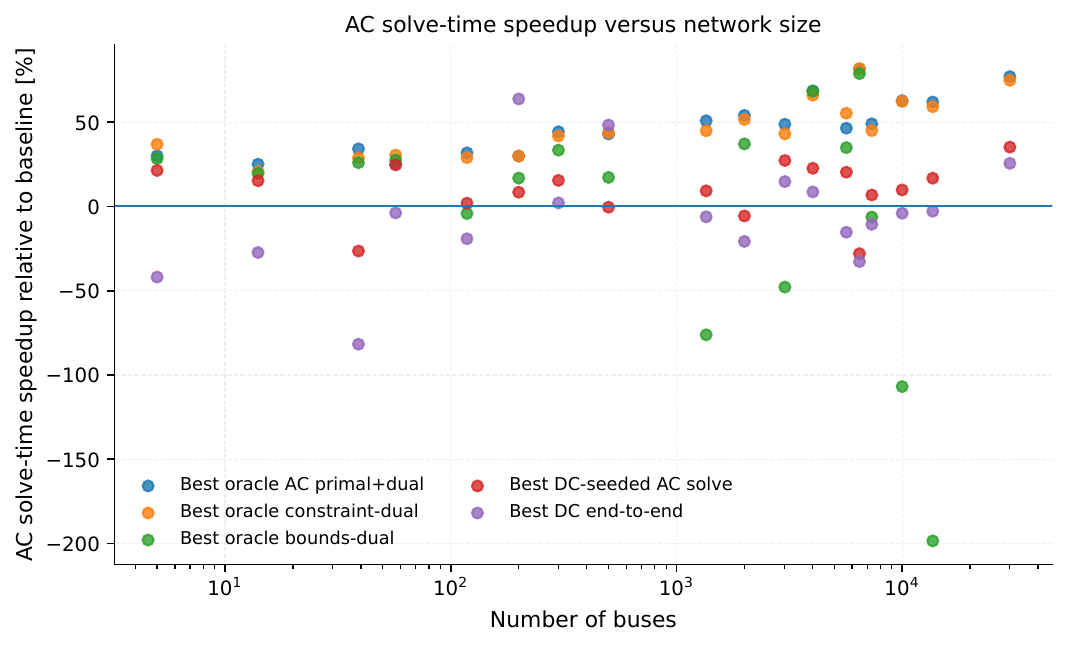}%
}{%
  \IfFileExists{figs/fig_07_speedup_vs_case_size.png}{%
    \includegraphics[width=0.95\columnwidth]{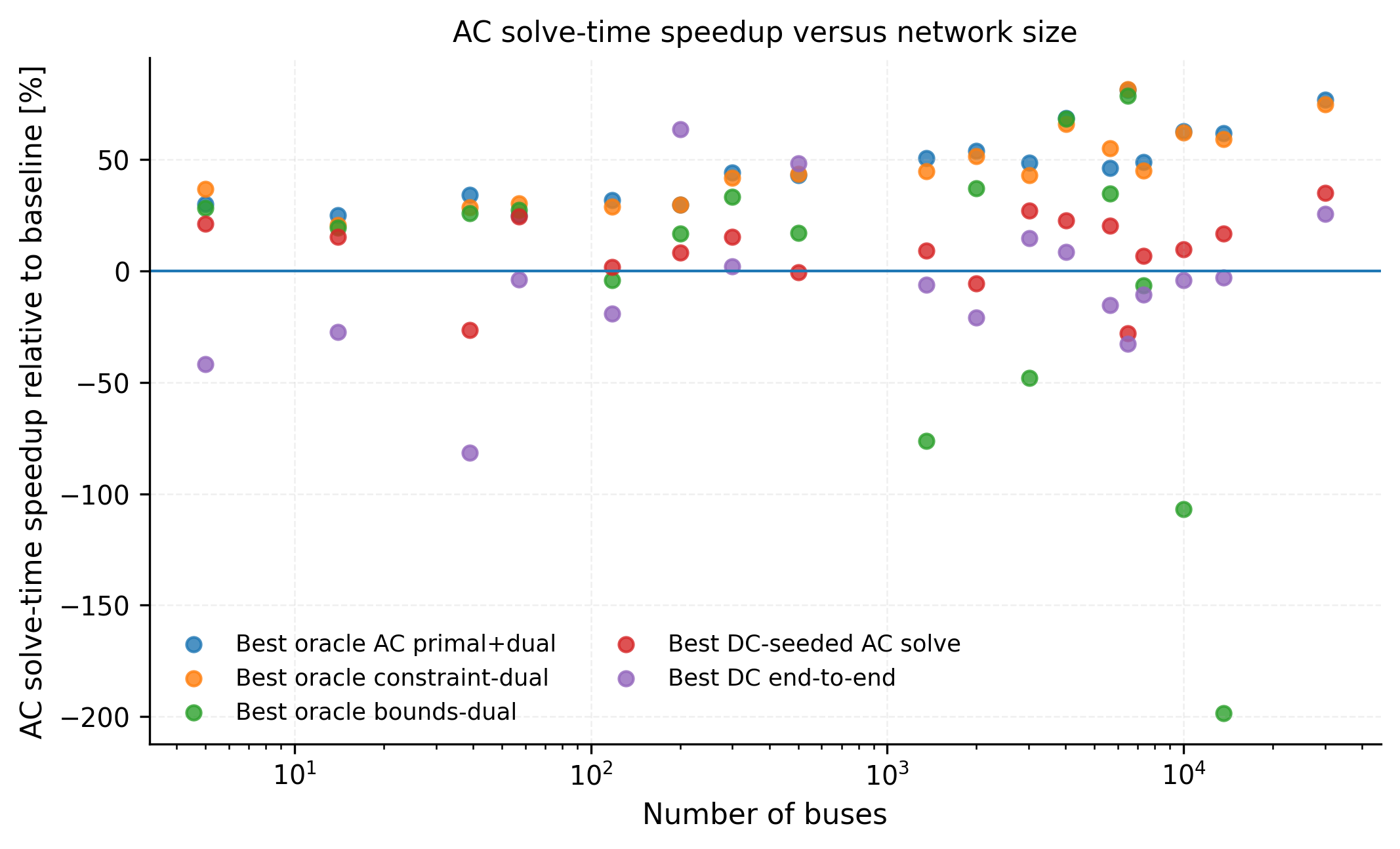}%
  }
}
\caption{AC solve-time speedup (\%) versus network size for the case-wise best envelope of each family.}
\label{fig:speedup_vs_case_size}
\end{figure}

\section{Static Combination Rankings}
\label{app:rankings}

Table~\ref{tab:combo_rank_oracle} provides the full fixed-policy rankings for all 15 combinations in each oracle family, and Table~\ref{tab:all_warmstarts_matrix} provides the case-level results matrix. In the O-PD family, the minimum-median-solve-time criterion selects a low-convergence policy, illustrating the survivor-bias problem discussed in Section~\ref{sec:global}. The full-vector $P_g{+}Q_g{+}V_m{+}V_a$ restart has a higher median solve time because it converges on all 18 baseline-convergent cases, but it remains the best practical fixed policy by convergence and within-family case wins.

\fi

\end{document}